\documentclass[11pt]{amsart}
\usepackage{amsmath,amssymb}
\usepackage{bm}
\usepackage{xcolor}
\usepackage{mathrsfs}
\usepackage{verbatim}
\usepackage{float}
\usepackage{setspace}
\usepackage{amsthm}
\usepackage{geometry}
\usepackage{color}

\newtheorem{myTheo}{Theorem}[section]
\newtheorem{mylem}{Lemma}[section]

\newtheorem{remark}{Remark}[section]
\newtheorem{Hypothesis}{Hypothesis}[section]

  \makeatletter
  \@addtoreset{equation}{section}
  \makeatother

\begin{document}

\title[3D $H^2$-nonconforming element]{3D $H^2$-nonconforming tetrahedral
    finite elements for the biharmonic equation}

\author {Jun Hu}
\address{LMAM and School of Mathematical Sciences, Peking University,
  Beijing 100871, P. R. China. }
\email{  hujun@math.pku.edu.cn}

\author {Shudan Tian}
\address{LMAM and School of Mathematical Sciences, Peking University,
  Beijing 100871, P. R. China. }
\email{ tianshudan@pku.edu.cn }

 \author {Shangyou Zhang}
\address{Department of Mathematical Sciences, University of Delaware,
    Newark, DE 19716, U.S.A. }
\email{  szhang@udel.edu }

\thanks{The first author was supported by  the NSFC Project 11271035,
  and  in part by the NSFC Key Project 11031006.}

% Received by the editors ?
%\date{ }

\begin{abstract}
In this article, a family of $H^2$-nonconforming finite elements on
  tetrahedral grids is constructed for solving the biharmonic equation in 3D.
In the family,  the $P_\ell$ polynomial space is enriched by
   some high order polynomials for all $\ell\ge 3$ and the corresponding
finite element solution converges at the optimal order $\ell-1$ in $H^2$ norm.
Moreover, the result is improved for two low order cases by using $P_6$ and $P_7$ polynomials to enrich
   $P_4$ and $P_5$ polynomial spaces,  respectively.
The optimal order error estimate is proved.
The numerical results are provided to confirm the theoretical findings.
\end{abstract}



\maketitle

\def\a#1{\begin{align*}#1\end{align*}}\def\an#1{\begin{align}#1\end{align}}
\def\ad#1{\begin{aligned}#1\end{aligned}}

\section{Introduction}
We consider the biharmonic equation:
\an{\ad{
  \Delta^2 u &=f \qquad \text{ in } \Omega,\\
  u=\partial_{\mathbf{n}}u& =0 \qquad \text{ on }  \partial \Omega, } \label{problem}
  }
where $\Omega$ is a bounded 3D polyhedral domain, and $\partial_{\mathbf{n}}u=\nabla u^T\mathbf{n},\mathbf{n}$ is the unit outer normal
  vector to $\partial \Omega$.
The weak formulation of (\ref{problem}) reads: Find $u\in H^2_0(\Omega)$ such that
\begin{equation}
	a(u,v)=(f,v)\qquad \forall v\in H^2_0(\Omega).
	\label{wproblem}
\end{equation}
Here $H_0^2(\Omega)=\{v\in H^2(\Omega)\mid v=\partial_\mathbf{n} v=0 ~ \text{ on }
           \partial\Omega   \}$ and $H^2(\Omega)$ is the standard Sobolev space
          \cite{Soblev}.
The bilinear form in \eqref{wproblem} is defined by
$$
a(u,v)=\int_{\Omega} D^2 u: D^2 v \text{d}\mathbf{x},\qquad
    (f,v)=\int_{\Omega}fv\text{d}\mathbf{x},
$$
where  $D^2u=(\frac{\partial ^2u}{\partial x_ix_j})_{i,j},
     1\leq i,j\leq 3$,  a $3\times 3$ tensor.

There are many numerical methods for the biharmonic equation (\ref{wproblem})
    such as the finite element method.
In a finite element method,  a finite dimensional space $V_h$ of piecewise polynomials
  is constructed to approximate $H_0^2$ functions.
If the finite element space is a subspace of  $H_0^2(\Omega)$, it is called a conforming finite element.
Otherwise, it is called a nonconforming element.
In a conforming finite element method, the subspace $V_h$ must be globally $C^1$.
One advantage of a conforming element is that the error of a numerical solution
  only depends on the approximation power of the finite element space.
But a globally $C^1$ differentiable element requires a high degree of polynomials.
On 2D triangular meshes, the lowest order conforming element is
   the Argyris $P_5$ element \cite{Argyris, Ciarlet}.
Such an element can be reduced a little to the Bell element \cite{Ciarlet, ShiWang, Shi-Wang}
   with 18 degrees of freedom by restricting a $P_4$ polynomial to
     a $P_3$ polynomial for normal derivatives on three edges of a triangle.
On 3D tetrahedral meshes,
a family of conforming elements of polynomials of degree 9 and above was constructed by Zhang
   \cite{Zhang-C1, Zhang-4D}.
On rectangular meshes in 2D and 3D, the problem is relatively simple.
The classic Bogner-Fox-Schmit(BFS) $C^1$-$Q_3$ element \cite{Ciarlet, Shi-Wang}
  can be easily to be extended to any higher degree,
   higher space dimension and higher smoothness \cite{Hu-Huang-Zhang, ZhangC1Qk}.
A minimal polynomial degree $C^m$-$Q_k$ conforming element
   on $n$-dimensional rectangular grids was also proposed by Hu and Zhang \cite{Hu-Huang-Zhang, Hu-Zhang}.

 However, the strong continuity requirement and the high degrees of freedom with higher order derivatives of conforming elements are not computationally desirable.
There have been many nonconforming elements developed.
The space of a nonconforming element, with fewer degrees of freedom on each element, is
    not a subspace of $C^1$ functions, and even not a subspace of $C^0$ functions.
The minimal degree nonconforming element for the biharmonic equation in 2D is the
    Morley element, with six degrees of freedom on each triangle, which was extended to any dimension in \cite{Wang-Xu}.
Like the Morley element, both the Veubake elements \cite{Fraeijs}
   and the NZT element \cite{Wang-Shi-Xu} are convergent
     with order $O(h)$ in an energy norm.
In a higher order nonconforming finite element method, the $P_\ell$ polynomial is usually
   enriched with higher order polynomials.
%Gao, Zhang and Wang \cite{Gao-Zhang-Wang} proposed
A second order method on 2D triangular meshes was proposed by Gao, Zhang and Wang \cite{Gao-Zhang-Wang}
    with two $P_5$ polynomials added to the $P_3$ polynomial space. In
Wang, Zu and Zhang \cite{Wang-Zu-Zhang}, the $P_3$ polynomial space was enriched by six $P_6$ polynomials, and in Chen, Chen and Qiao \cite{Chen-Chen-Qiao}, % proposed a nonconforming element,
     four $P_6$ polynomials, four $P_7$ polynomials and four $P_8$ polynomials
    to achieve a second order nonconforming element in 3D.
A family 3D  elements were constructed in \cite{Guzman}, using $P_{\ell+5}$ polynomials
  to enrich the $P_{\ell}$ polynomial space.

Recently,  a new estimate technology was proposed in \cite{Hu-Zhang-3D} by Hu and Zhang,
   generalizing the ideas of \cite{Hu-Ma-Shi,Hu-Zhang-T}.
 %    a side-patch projection based on error analysis.
The error estimate is based on two continuity hypotheses on the gradient jump and
  the function value jump across $d-1$ dimensional internal sides.
%For $\ell-1$-th order method, it only need suppose
%$\ell-3$-order function continuity and $\ell-2$ gradient continuity when %cross the element. Therefore, the element is not in $C^0$.
The theory was applied to construct a second order nonconforming element on
tetrahedral grids, enriching the $P_3$ polynomial space by eight $P_4$ polynomials on each tetrahedron. This results in the lowest polynomial degree element of second order approximation in 3D so far.
Compared with other $H^2$-nonconforming elements,  that element does not require vertex
   continuity.
The aim of this work is to extend the $P_3$ nonconforming element to a family of
   $P_\ell$ nonconforming elements for all $\ell$.
For large $\ell$, it is shown that the minimum polynomial degree of enriched polynomials is $\ell+4$,
  and such a family of $H^2$-nonconforming finite elements is desired on tetrahedral
   meshes.
It is noted that the polynomial degree of the elements in \cite{Guzman} is
   one degree higher than that of the elements in this work.
But for small $\ell$,  $\ell+4$ is not the minimum polynomial degree  for the optimal order of convergence.
For example, when $\ell=3$,  $\ell+1$ is the minimum degree as shown in \cite{Hu-Zhang-3D}.
In this work, for the $\ell=4$ and $\ell=5$ cases, the $P_4$ and $P_5$ polynomial spaces are enriched by
    $P_6$, $P_7$ polynomials,  respectively.
These are the lowest degree of enriched polynomials that can be found for these two cases so far.

The rest of the article is organized as follows. In section 2, we introduce two hypotheses and
    present an optimal energy-norm estimate based on the two hypotheses.
In section 3, we construct a family of $H^2$-nonconforming elements for all polynomial
  degrees on tetrahedral grids.
In section 4, a lower order polynomials is used to replace $P_7$ and $P_8$ polynomials for the third and fourth
   order element methods and all explicit basis functions for the two elements are given in this section and appendix.
Finally, we present some numerical results to confirm the theoretical results.

\section{Hypotheses and abstract theory}
  Let $\mathcal{T}_h=\{T\}$
   be a regular tetrahedral grid on $\Omega$, cf. \cite{Brenner-Scott} and $h$ is the mesh size of $\mathcal{T}_h$.
  Let $F$ and  $e$ be a two dimensional face triangle and a one dimensional
  edge of element $T$, respectively. Let
 $P_{\ell}(G)$ represent the space of polynomials of degree less than or equal to $\ell$ over $G$.
  Let $\mathcal{F}_h$ be the set of all two dimensional face triangles of $\mathcal{T}_h$. Let $\omega_F$ be the union of two elements sharing the two dimensional face triangle $F$.
  Given integer $\ell>0$,  let $V_{h,\ell}$ be the nonconforming element  space of $H_0^2(\Omega)$ on
     the mesh $\mathcal{T}_h$, defined in \eqref{Vh} below. Let
       $\Pi_{\ell,G}$ be the $L^2$ projection operator onto $P_{\ell}(G)$. %and
 %     $\Pi_{\ell-1,T}$ be the $L^2$ projection operator onto $P_{\ell-1}(T)$.\\
     
 The finite element problem, discretizing the biharmonic
 equation(\ref{problem}), is:
   Find
$u_h\in V_{h,\ell}$ such that
\begin{equation}
	(D^2u_h,D^2v_h)_h=(f,v_h)\quad \forall v_h\in V_{h,\ell},
  \label{dproblem}
\end{equation}
where the discrete inner product is defined as $(\cdot,\cdot)^2_h=\sum_{T\in \mathcal{T}_h}(\cdot,\cdot)^2_T$.
The existence and uniqueness of solutions in problem (\ref{dproblem}) follow from
  the norm $|\cdot|_{2,h}=(D^2\cdot,D^2\cdot)_h^{1/2}$ (to be proved)
    on $V_{h,\ell}$.
By the second Strang's Lemma \cite{Shi-Wang}, we have
$$
|u-u_h|_{2,h}\leq C \inf_{v_h\in V_{h,\ell}}|u-v_h|_{2,h}
  +\sup_{0\neq w_h\in V_{h,\ell}}\frac{|(f,w_h)-(D^2u,D^2w_h)_h|}{|w_h|_{2,h}}.
$$
To bound the error in the second term, i.e., the consistency error,
    the following two hypotheses were proposed in \cite{Hu-Zhang-3D}.
\begin{Hypothesis}
	For all internal face triangles $F$ of $\mathcal{T}_h$,  assume
\begin{equation} \label{H1}
	\int_F [\nabla_h v_h]\cdot q\mathrm{d} S =0 \quad
   \forall q\in (P_{\ell-2}(F))^3 \text{ \ and \ }
       \forall v_h\in V_{h,\ell},
\end{equation}
where $[\cdot]$ is the jump across $F$, $\nabla_h$ and $D_h^2$ are the discrete counterpart of $\nabla$ and $D^2$, respectively, defined element wise.
For all domain boundary face triangles $F$ of $\mathcal{T}_h$, assume
\begin{equation} \label{H12}
\int_F\nabla_h v_h\cdot q\mathrm{d} S=0 \quad
    \forall q\in (P_{\ell-2}(F))^3 \text{ \ and \ } \forall v_h\in V_{h,\ell}.
\end{equation}
\label{hyp1}
\end{Hypothesis}

\def\b#1{\mathbf{#1}}\def\r#1{\mathrm{#1}}

\begin{Hypothesis}
	For all internal face triangles $F$ of $\mathcal{T}_h$,  assume
	\begin{equation} \label{H2}
	\int_F [v_h]q\mathrm{d} S=0 \quad \forall q\in P_{\ell-3}(F)
      \text{ \ and \ } \forall v_h\in V_{h,\ell}.
	\end{equation}
For all domain boundary face triangles $F$ of $\mathcal{T}_h$, assume
		\begin{equation}  \label{H22}
		\int_F v_hq\mathrm{d} S=0
   \quad \forall q\in P_{\ell-3}(F) \text{ and } \forall v_h\in V_{h,\ell}.
		\end{equation}
\label{hyp2}
\end{Hypothesis}

\begin{myTheo}\label{Hu}
[Theorem 2.1 in \cite{Hu-Zhang-3D}]   Assume $V_{h,\ell}$ satisfies Hypotheses \ref{hyp1} and
     \ref{hyp2} with $\ell\geq 3$, and that the seminorm
     $\|D^2_h\cdot\|_{0}=(D^2\cdot,D^2\cdot)_h^{1/2}$ defines a norm  over the nonconforming finite element space $V_{h,\ell}$.
   Let $u_h$ and $u$ be the solution of (\ref{dproblem}) and (\ref{problem}), respectively.
   Then
\begin{align*}
\|D^2_h(u-u_h)\|_{0}&\leq C \inf_{s_h\in V_{h,\ell}}\|D^2_h(u-s_h)\|_{0}+
C\left(\sum_{F\in\mathcal{F}_h}\|(I-\Pi_{\ell,\omega_F})D^2u\|^2_{0,\omega_F}\right)^{1/2}\\
   &\quad \ + \left(\sum_{T\in\mathcal{T}_h}h^4\|(I-\Pi_{\ell-1,T})f\|^2_{0,T}\right)^{1/2}.
\end{align*}
%where $D^2_h$ is an element-wise tensor of $D^2$,
%   $\|D^2_h u_h\|_0^2=\sum_{T\in \mathcal{T}_h}\|D^2u_h\|_{0,T}^2$, and

\end{myTheo}

\section{A family of $H^2$ non-conforming finite elements}
 In this section, we construct a family of $H^2$-nonconforming
   finite element spaces  for the biharmonic problem in 3D.
   Based on Hypothesis \ref{hyp1} and \ref{hyp2},
 the dual basis of the finite element space in 3D, i.e., the degrees of freedom of the finite element,
  consists of
\begin{align}
&\mathbb{E}^{(\ell)}(v)= \frac{1}{|e|}\int_ev P_{\ell-2}(e)\mathrm{d}s \ \text{ on all edges};
\label{degree1}\\
&\mathbb{F}^{(\ell)}(v)=\frac{1}{|F|}\int_Fv P_{\ell-3}(F)\mathrm{d} S \ \text{ on all face triangles};
\label{degree2}\\
&\mathbb{T}^{(\ell)}(v)= \frac{1}{|T|}\int_Tv P_{\ell-4}(T)\mathrm{d}\mathbf{x}
      \ \text{ on all tetrahedrons};
\label{degree3}\\
&\mathbb{N}^{(\ell)}(v)=
\frac{1}{|F|}\int_F\partial_{\mathbf{n}}v P_{\ell-2}(F)\mathrm{d}S,
    \ \text{ on all face triangles}.
\label{degree4}
\end{align}
When $\ell=2$, (\ref{degree2}) and (\ref{degree3}) drop and this element is
  the 3D Morley element.  When $\ell=3$, (\ref{degree3}) drops. Therefore, in this article, we only consider the cases $\ell\geq 3.$

What is the dimension of the dual basis \eqref{degree1}--\eqref{degree4}?
How many high order polynomials are needed to enrich each polynomial space $P_{\ell}(T)$ so that
  the enriched space can fulfill \eqref{degree1}--\eqref{degree4}?
What is the minimum degree of enriched polynomials for the base polynomial space $P_{\ell}(T)$?
These questions would lead to a proper definition of the enriched space $\mathcal{P}_{\ell}^+(T)$ below. Next, we define the enriched polynomial space $\mathcal{P}_\ell^+(T)$.

% expand((6*(l-1)+4*(l-2)*(l-1)/2 +4*(l-0)*(l-1)/2 +(l-3)*(l-2)*(l-1)/6)*6)
% expand((l+4)*(l+5)*(l+6));
% factor((6*(l-1)+4*(l-2)*(l-1)/2 +4*(l-0)*(l-1)/2 +(l-3)*(l-2)*(l-1)/6)-(l+1)*(l+2)*(l+3)/6)

\begin{mylem}
	For $\ell\geq 27$, the number of local degrees of freedom of
   \eqref{degree1}--\eqref{degree4} is bigger than the  dimension of
	$P_{\ell+3}(T)$.
	\label{dim}
\end{mylem}
\begin{proof}
	The number of degrees of freedom of (\ref{degree1}) to (\ref{degree4}) is
	$\frac{1}{6}
	(\ell^3+18\ell^2-\ell-18)$
	which is bigger than the dimension  $\frac{1}{6}(\ell^3+15\ell^2+74\ell+120)$ of $P_{\ell+3}$ when $\ell \geq 27$.
\end{proof}

Therefore,  the minimum polynomial degree of enriched polynomials is $\ell+4$ for
  the $P_{\ell}$ polynomial space when $\ell\geq 27$. The difference between the number of degrees of freedom of (\ref{degree1})--(\ref{degree4})
  and the dimension of $P_{\ell}(T)$ is
\a{ &\quad \ \frac{1}{6}
	(\ell^3+18\ell^2-\ell-18) - \frac 16 (\ell^3+6\ell^2+11\ell+6) \\
    & = (2 \ell^2 -2 \ell)- 4 \\
    & = 4 (\ell-1)(\ell-0)/2 - 4 \\ & = 4 \dim P_{\ell-2, 2D} - 4.  }
Thus, the $P_{\ell}$ polynomial space has to be enriched by $\dim P_{\ell-2} (F) - 1 $ high order polynomials
   on each face triangle.

%We first study such $(\dim P_{\ell-2} (F))$ face bubble functions.
%Then we eliminate one such bubble function when defining the enriched space $\mathcal{P}^+_\ell$.

\begin{mylem}
Let $F_m$ be a face triangle of $T$.  A function  $v\in b_{F_m}^2P_{\ell-2}(T)=
         \{\lambda_i^2\lambda_j^2\lambda_k^2q~|~q\in P_{\ell-2}(T) \}$,
    where the face triangle $F_m$, the $m$-th face ,is formed by three vertices $i$, $j$ and $k$,  and
   $\lambda_i$ is a linear function valued 1 at vertex $i$ and $0$ at the rest vertices,
is unisolvent by the following degrees of freedom,

\an{ \label{3dof}  (\mathrm{i})&\int_{F_m}v P_{\ell-3}(F_m)\mathrm{d}S,\\ ~(\mathrm{ii})&
 \int_{T}v P_{\ell-4}(T)\mathrm{d}\mathbf{x},\\~
(\mathrm{iii})&\int_{F_m}\partial_{\mathbf{\b n}_m}v  P_{\ell-2}(F_m)\mathrm{d}S.
}
When $\ell=3$, (ii) of \eqref{3dof} drops.
\begin{comment}
\begin{equation}
v\mapsto\left\{
\begin{cases}
\int_{F_m}v P_{\ell-3}(F_m)\mathrm{d}S&(\mathrm{i})\\
\int_{T}v P_{\ell-4}(T)\mathrm{d}\mathbf{x}&(\mathrm{ii})\\
\int_{F_m}\partial_{\mathbf{\b n}_m}v  P_{\ell-2}(F_m)\mathrm{d}S&(\mathrm{iii})
\end{cases}
\right.
\end{equation}
\end{comment}

\end{mylem}

\begin{proof}
	Without loss of generality, we prove it on $F_4$.
Let $\phi=b^2_{F_4}q \in b^2_{F_4}P_{\ell-2}(T)$ and suppose it vanishes on all degrees of freedom of
  \eqref{3dof} (i),(ii) and (iii).
  The number of these degrees of freedom is
$$\mathrm{dim}P_{\ell-4}(T)+\mathrm{dim}P_{\ell-3}(F)+\mathrm{dim}P_{\ell-2}(F)=\mathrm{dim}P_{\ell-2}(T).$$
	Therefore, we only need to prove $\phi\equiv 0$.
 At first we want to prove if $\phi$ vanishes on all degrees of freedom then $\phi$ satisfies
	\begin{equation}
	\int_{F_4}(\mathbf{r}\cdot \nabla \phi) q_{\ell-2}\mathrm{d} S=0 \qquad\forall q_{\ell-2}\in
          P_{\ell-2}(F_4),
	\label{normaldecomp}
	\end{equation}
for all $\mathbf{r}\in \mathbb{R}^3$.
Let $\mathbf{t}$ be the tangential vector, the projection of $\mathbf r$
    on the face $F_4$,  $\mathbf{t}=\mathbf{r}-(\mathbf{r}\cdot \b n_4)\b n_4 $.
Here $\b n_4$ is the unit outer normal to the face $F_4$.
It follows from (\ref{3dof})(iii) that
	\begin{align*}
	\int_{F_4}(\mathbf{r}\cdot \nabla \phi) q_{\ell-2}\mathrm{d}S
  &=\int_{F_4}((\b t +(\mathbf{r}\cdot \b n_4)\b n_4)\cdot \nabla \phi) q_{\ell-2}\mathrm{d}S\\
	&=\int_{F_4}(\b t\cdot \nabla \phi) q_{\ell-2}\mathrm{d}S
	+\int_{F_4}((\mathbf{r}\cdot \b n_4)\b n_4 \cdot \nabla \phi) q_{\ell-2}\mathrm{d}S\\
    &=\int_{F_4}(\b t\cdot \nabla \phi) q_{\ell-2}\mathrm{d}S.
	\end{align*} 	
Noting $\phi=b_{F_4}^2 q$ that $ (\phi q_{\ell-2})|_{\partial F_4}=\b 0$, we have, by \eqref{3dof}(i),
\a{
	\int_{F_4}(\b t\cdot \nabla \phi) q_{\ell-2}\mathrm{d}S
    &=\int_{F_4} \partial_{\b t} (\phi  q_{\ell-2}) \mathrm{d}S
       -\int_{F_4}\phi(\b t\cdot \nabla q_{\ell-2})\mathrm{d}S \\
    &= -\int_{F_4}\phi(\b t\cdot \nabla q_{\ell-2})\mathrm{d}S=0,
	} since $\b t\cdot \nabla q_{\ell-2}\in P_{\ell-3}(F_4)$.
	Therefore, we have proved equation (\ref{normaldecomp}).
 Next, let $\mathbf{r}$ in \eqref{normaldecomp} be
\a{  \mathbf{r}= c(\nabla\lambda_2 \times \nabla\lambda_3),~
        c=1/(\nabla\lambda_2 \times \nabla\lambda_3)\cdot\nabla\lambda_1, }
  where the box product is nonzero, a scaled volume of the tetrahedron $T$.
	\a{
    \nabla\phi\cdot \mathbf{r}
	  &= (2 b_{F_4} q \nabla b_{F_4} + b_{F_4}^2 \nabla q) \cdot \b r \\
       &= 2 b_{F_4} q (\lambda_2\lambda_3 \nabla \lambda_1+\lambda_1\lambda_3 \nabla \lambda_2
           +\lambda_1\lambda_2 \nabla \lambda_3) \cdot \b r
            + b_{F_4}^2 \nabla q \cdot \b r \\
       &= 2 b_{F_4} q  \lambda_2\lambda_3
            + b_{F_4}^2 \nabla q \cdot \b r.  }
From (\ref{normaldecomp}) we have
	\begin{align*}
	0 & = \int_{F_4}(\mathbf{r}\cdot \nabla \phi)q_{\ell-2} \mathrm{d}S
	\\& =
    \int_{F_4} \phi \frac{2q_{\ell-2}}{\lambda_1} \mathrm{d}S +
	\int_{F_4}  b_{F_4}^2  (\nabla q \cdot \b r)  q_{\ell-2} \mathrm{d}S
	\\& =
    \int_{F_4} \phi  q_{\ell-3}  \mathrm{d}S +
	\int_{F_4} \frac{\lambda_1^3\lambda_2^2\lambda_3^2}2   (\nabla q \cdot \b r)  q_{\ell-3} \mathrm{d}S
     \\ &= \int_{F_4} \frac{\lambda_1^3\lambda_2^2\lambda_3^2}2
          (\nabla q \cdot \b r)  q_{\ell-3} \mathrm{d}S ,
	\end{align*}
where we choose $q_{\ell-2}=\lambda_1 q_{\ell-3}/2$ for an arbitrary $q_{\ell-3}\in P_{\ell-3}(F_4)$,
   and consequently the first term is zero due to \eqref{3dof}(i).
Since $ \frac{\lambda_1^3\lambda_2^2\lambda_3^2}2>0$ on $F_4$ except at $\partial F_4$,
     the $P_{\ell-3}(F_4)$ polynomial vanishes,
\an{ \label{1orth} (\nabla q \cdot \b r) |_{F_4} = 0.  }
Repeating above arguments with $\b r=c(\nabla\lambda_3 \times \nabla\lambda_1)$ and
    $\b r=c(\nabla\lambda_1 \times \nabla\lambda_2)$, we obtain \eqref{1orth} in the other
   two linearly independent directions and consequently
\an{ \nabla q |_{F_4} = \b 0,  }
which implies $\partial_{\b n_4} q |_{F_4} = 0$ and $\partial_{\b t} q |_{F_4} = 0$
   for any tangential vector $\b t$ on the face $F_4$.
Thus, $q|_{F_4}$ is a constant.  By \eqref{3dof}(i) this constant is $0$. Therefore, for $\ell=3$ we have $q=0$. 
If $\ell\geq 4$, we obtain
\a{ q=\lambda_4^2 q_0 \quad \hbox{ for some } \ q_0 \in P_{\ell-4}(T). }
By \eqref{3dof}(ii), $q_0=0$. Thus $q=0$.  The proof  is completed.
\end{proof}

The nodal basis of space $b_{F_4}^2 P_{\ell -2}(T)$, dual to the degrees of freedom of \eqref{3dof}(i)--(iii), is
\an{ \nonumber
    &\phi^{(1)}_{1,F_m}, \phi^{(1)}_{2,F_m}, \dots, \phi^{(1)}_{(\ell-2)(\ell-1)/2,F_m}, \\
     \nonumber
     &\phi^{(2)}_{1,F_m}, \phi^{(2)}_{2,F_m}, \dots, \phi^{(2)}_{(\ell-3)(\ell-2)(\ell-1)/6,F_m}, \\
    \label{norbase}
      &\phi_{1,m}, \phi_{2,m}, \dots, \phi_{d_0,m},\phi_{d_0+1,m}, }
where $d_0=(\ell-1)(\ell-0)/2-1$ and the last basis function satisfies
\an{\label{drop} \frac{1}{|F_m|}\int_{F_m}\partial_{\mathbf{\b n}_m}\phi_{d_0+1,m}\mathrm{d}S=1. }
Except the last basis function $\phi_{d_0+1,m}$, we enrich the $P_{\ell}$ polynomial space by the third group basis functions of \eqref{norbase} on the
  four faces:
\an{
	\mathcal{P}_\ell^+(T) = P_{\ell}(T)+\operatorname{span}_{1\le m\le 4} \{
                   \phi_{1,m},  \dots, \phi_{d_0,m} \}.
\label{Vhl} }
Noting that each $\phi_{i,m}(=b_{F_m}^2 q_{\ell-2})$ above
        vanishes on six edges of $T$ and
         it is a nodal basis function
   of $\mathcal{P}_{\ell}^+(T)$ associated to the fourth group degrees of freedom of \eqref{degree4}.
   
   Then the family of nonconforming finite element spaces is defined via the local space $\mathcal{P}_{\ell}^+(T)$.
   \begin{align}
   \label{Vh}
   V_{h,\ell}&=\Big\{v\in L^2(\Omega) \mid  v|_T\in \mathcal{P}^+_\ell(T),
   \int_e v p_{\ell-2}\text{d}s ~\text{is continuous at internal}\\
   \nonumber &\quad  \text{edges of }\mathcal{T}_h, \ \int_F vp_{\ell-3}\text{d}S ~
   \text{and}~ \int_F \partial_{\mathbf{n}}v p_{\ell-2}\mathrm{d} S
   \text{ are continuous }\\
   \nonumber  &\quad  \text{on internal face triangles of } \mathcal{T}_h,
   \quad \int_e v p_{\ell-2}\text{d} s=\int_F vp_{\ell-3}\text{d} S =\\
   \nonumber  &\quad \int_F \partial_{\mathbf{n}}v p_{\ell-2}\text{d}S =0,
   \text{ at boundary edges and on face triangles of }\mathcal{T}_h
   \Big\},
   \end{align}
   where $\mathcal{P}^+_\ell(T)$ is defined in \eqref{Vhl} and $p_{\ell}$ denotes a
   general $P_{\ell}$ polynomial.
\begin{myTheo}
The shape function space $\mathcal{P}^+_\ell(T)$
is unisolved by degrees of freedom of
(\ref{degree1})-(\ref{degree4}).
\label{unique}
\end{myTheo}

\begin{proof}
First,  (\ref{degree1})-(\ref{degree3}) and the four vertex valuations $\{ v(\b x_m) \}$    form a dual
  basis for $P_{\ell}(T)$.
This can be verified as follows.  If all the degrees of freedom of $v\in \mathcal{P}_\ell(T)$ vanish,
  then (1) $v$ vanishes on each edge because it vanishes at two end points of the edge and its moments of order $\leq\ell-2$ on the edge
  also vanishes;
  (2) $v$ vanishes on each face triangle because it vanishes on the 3 edge of the face triangle and its moments of order $\leq \ell-3$
  on the face triangle vanishes too;
  (3) $v$ vanishes on the tetrahedron $T$ because it vanishes on the four face triangles of $T$ and its moments of
   order $\ell-4$ on $T$ also vanishes.
  %(1) $v$ vanishes on each edge because its 2 end-edge values and its $\ell-2$ order edge moments vanish;
  %     (2) $v$ vanishes on each face triangle because its 3 edge values and its $\ell-3$ order
  %        face moments vanish;
  %     (3) $v$ vanishes on the tetrahedron $T$ because its 4 face triangle values an its
  %      $\ell-4$ order volume moments vanish.

The corresponding basis functions of $P_{\ell}(T)$, dual to (\ref{degree1})-(\ref{degree3})
    and the following four degrees of freedom 
\begin{equation}
	\frac{1}{|F_m|}\int_{F_m}\partial_{\mathbf{n}_m}v\mathrm{d}S,~m=1,\cdots,4
	\label{constdof}
\end{equation}  
    
    are
    %$\{ v(\b x_1), v(\b x_2)$, $v(\b x_3)$, $v(\b x_4) \}$,
\an{ \label{Pl}
   \phi_{1,P_{\ell}},\ ...,\ \phi_{d_1,P_{\ell}},
      \phi_{d_1+1,P_{\ell}},\ \phi_{d_1+2,P_{\ell}},\
   \phi_{d_1+3,P_{\ell}}, \ \phi_{d_1+4,P_{\ell}}, }
where $d_1=\dim P_{\ell}(T) - 4$. Note that the last four basis functions 
vanish for \eqref{degree1}-\eqref{degree3} and satisfy the following orthogonal property with respect to the degrees of freedom 
of \eqref{constdof}, i.e.,
\a{ \int_{F_m} \partial_{\b n_m} \phi_{d_1+i,P_{\ell}} \r dS
      =\begin{cases} 0 & \hbox{if } i \ne m, \\
              \hbox{non-zero}  & \hbox{if } i = m, \end{cases} \quad i,m=1,2,3,4. }
 In fact we have the following expression for such a function, independent of $\lambda_1, \lambda_2$ and
                  $\lambda_3$,
\an{\label{se}
   \phi_{d_1+4, P_\ell} = \lambda_4 ( a_1 + a_2 \lambda_4+ \dots + a_{\ell} \lambda_4^{\ell-1}),
  }
where $a_1 \ne 0$, $a_1+a_2+\dots+a_{\ell}=1$.  This can be verified by applying degrees of freedom of
    (\ref{degree1})-(\ref{degree3}) to it.
\begin{comment}
	We will show the last four functions form, defined in (\ref{nodalbase}), an orthogonal basis for the dropped 4 dof of group 4 dof
	\eqref{degree4}, defined in    \eqref{drop}, i.e.,
	\a{ \int_{F_m} \partial_{\b n_m} \phi_{\ell_1+i,P_{\ell}} \r dS
		=\begin{cases} 0 & \hbox{if } i \ne m, \\
			\hbox{non-zero}  & \hbox{if } i = m, \end{cases} \quad i,m=1,2,3,4. }
\end{comment}
\a{   \frac{1}{|F_m|}\int_{F_4} \partial_{\b n_4} \phi_{d_1+4,P_{\ell}} \r dS
     = -|\nabla \lambda_4|a_1 |F_4| \ne 0, }

  where $|F_4|$ is the area of the $F_4$.
But on any other face triangle, say $F_1$, we have
\a{   \int_{F_1}  \partial_{\b n_1} \phi_{d_1+4,P_{\ell}}\r  dS
    & = \int_{F_1}  {\b n_1}\cdot \nabla \phi_{d_1+4,P_{\ell}} \r dS \\
    & = \int_{F_1}  \Big( \frac{\b n_2\times \b n_4} {(\b n_2\times \b n_4)\cdot \b n_1}
        - \frac {c_1 \b t_1} {(\b n_2\times \b n_4)\cdot \b n_1} \Big)
               \cdot \nabla \phi_{d_1+4,P_{\ell}}\r  dS \\
     & = \int_{F_1}   \frac{(\b n_2\times \b n_4)\cdot \nabla \phi_{d_1+4,P_{\ell}} }
       {(\b n_2\times \b n_4)\cdot \b n_1}\r  dS \\
  & = \int_{F_1}   \frac{ \b n_2  \cdot (\b n_4 \times \nabla \phi_{d_1+4,P_{\ell}})  }
       {(\b n_2\times \b n_4)\cdot \b n_1}\r  dS \\
  & = \int_{F_1}   \frac{ \b n_2  \cdot \b 0  }
       {(\b n_2\times \b n_4)\cdot \b n_1}\r  dS = 0,
} where $\b t_1 $ is a tangent vector on $F_1$ such that $\b n_2\times \b n_4= c_0\b n_1 + c_1 \b t_1
    = [(\b n_2\times \b n_4)\cdot \b n_1]\b n_1 + c_1 \b t_1,$    $(\b n_2\times \b n_4)\cdot \b n_1=6|T|\ne 0$,
    $\int_{F_1} \b t_1 \cdot \nabla \phi_{d_1+4,P_{\ell}}\r  dS=
     \int_{\partial F_1} (\b t_1 \cdot \b t_{\partial F_1})
            \phi_{d_1+4,P_{\ell}}\r  dS =0$ by \eqref{degree1},
 and %$\b n_4 \parallel \nabla \phi_{\ell_1+4,P_{\ell}} $
 $\mathbf{n}_4$ is parallel to $\nabla \phi_{d_1+4,P_{\ell}} $
 everywhere by \eqref{se}.

We now show that these functions from \eqref{Pl} and \eqref{norbase} form a (not dual to the degrees of freedom of \eqref{degree1} to \eqref{degree4}) basis of
  $\mathcal{P}^+_{\ell}(T)$.
We only need to show they are linearly independent. Assume that
\a{ u = \sum_{i=1}^{d_1} a_i \phi_{i, P_\ell} +
       \sum_{i=1}^{4} b_i \phi_{d_1+i, P_\ell} +
      \sum_{m=1}^{4} \sum_{i=1}^{d_0}  c_{m,i} \phi_{ i,m} =0 }
for parameters $a_i,b_i$ and $c_{m,i}.$
%    where $d_0$ and $d_1$ are defined in \eqref{norbase} and \eqref{Pl} respectively.
Sequentially,
\a{ \ad{&\hbox{applying degrees of freedom of \eqref{degree1}--\eqref{degree3} to $u$} &&\Rightarrow a_{i}=0, \\
    &\hbox{applying four degrees of freedom of \eqref{constdof} to $u$} &&\Rightarrow b_i=0, \\
    &\hbox{applying the rest degrees of freedom of \eqref{degree4}  to $u$} &&\Rightarrow c_{m,i} =0. } }
This completes the proof.
\end{proof}
   Finally we present a convergence theorem of the family of finite elements.
\begin{myTheo}
The equation (\ref{dproblem}) has a unique solution $u_h \in V_{h,\ell}$. Moreover, the $H^2$ semi-norm error estimate with $u\in H^{\ell+1}(\Omega)\cap H^2_0(\Omega)$ is given by,
$$
\|D_h^2(u-u_h)\|_0\leq Ch^{\ell-1}|u|_{\ell+1}.
$$
\label{crate}
\end{myTheo}

\begin{proof} Let $u_h$ be a solution to \eqref{dproblem} with $f=0$ there.
  On each tetrahedron, $D^2 u_h=0$ which implies that $\nabla_h u_h$ is a piecewise constant vector there.
The jump condition \eqref{hyp1} indicates that $\nabla_h u_h$ is a global constant vector.
 By the normal derivative boundary condition, we have $\nabla_h u_h=\b 0$. Thus,
$u_h$ is a piecewise constant on each tetrahedron.
 Combination with the jump condition \eqref{hyp2}, we get $ u_h$ is a global constant.
 Then, the function value boundary condition implies $u_h=0$.
  Thus, the square linear system of equations  \eqref{dproblem} has a unique solution.

By \eqref{degree1}--\eqref{degree4},  Theorem \ref{Hu}, and the standard interpolation theory,
  the theorem is proved.

\end{proof}

\section{The lower order situation}
 According to Lemma \ref{dim}, $P_{\ell+4}$ polynomials  are needed for enrichment when $\ell \geq 27$.
 \begin{comment}
 {\color{red} old version:
We can use lower order bubble functions for lower degree polynomials $P_\ell$.
 In this section,we  consider the lower order cases.
We present a $P_4+B_6$  element  and a $P_5+B_7$ element for (\ref{dproblem}).
Here $B_6 \subset P_6(T)$ and $B_7 \subset P_7(T)$.
$F_m = (i,j,k)$ represents the $m$-th face triangle with vertexes $i, j$ and $k$ in the following order
\begin{equation}
F_1=(2,3,4),~F_2=(3,4,1),~F_3=(4,1,2),~F_4=(1,2,3).
\label{faceorder}
\end{equation}
}
\end{comment}
 However, when $\ell$ is small, the enriched $P_{\ell}$ polynomial space, defined in \eqref{norbase} , can be improved by using lower degree polynomials compared with $P_{\ell+4}$ polynomials. In this section, we mainly focus on finding the optimal degree of such a enriched $P_{\ell}$ function space for $\ell=4,5$. 
 
 For convenience, we use $\mathbb{N}^{(\ell)}_{m,l},m=1,\cdots,4$ to represent the $l$-th normal derivative degree of freedom on the $m$-th face, i.e.,
 $$
 \mathbb{N}^{(\ell)}_{m,l}(v)=\frac{1}{|F_m|}\int_{F_m}\partial_{\mathbf{n}}(v)\lambda_i^{l_1}\lambda_j^{l_2}\lambda_k^{l_3}\mathrm{d}s,
 ~l=1,\cdots,\mathrm{dim}(P_{\ell-2}(F_m)),
 $$
 where $l_1+l_2+l_3=\ell-2$ and the subscript
  $l$ of $\mathbb{N}^{(\ell)}_{m,l}$ represents that $\lambda_i^{l_1}\lambda_j^{l_2}\lambda_k^{l_3}$ is the $l$-th basis function on $P_{\ell-2}(F_m)$. Here, we suppose that the face $F_m$ consists of vertices $i,j,k.$  Similarly, we can define $\mathbb{E}^{(\ell)}_{m,l},\mathbb{F}^{(\ell)}_{m,l},\mathbb{T}^{(\ell)}_{l}$ as
\begin{align*}
 \mathbb{E}^{(\ell)}_{m,l}(v)&=\frac{1}{|e_m|}\int_{e_m}\partial_{\mathbf{n}}(v)\lambda_i^{l_1}\lambda_j^{l_2}\mathrm{d}s,~m=1,\cdots,6,~l=1,\cdots,\mathrm{dim}(P_{\ell-2}(e_m)),~l_1+l_2=\ell-2;\\
  \mathbb{F}^{(\ell)}_{m,l}(v)&=\frac{1}{|F_m|}\int_{F_m}\partial_{\mathbf{n}}(v)\lambda_i^{l_1}\lambda_j^{l_2}\lambda_k^{l_3}\mathrm{d}S,~
  m=1,\cdots,4,~l=1,\cdots,\mathrm{dim}(P_{\ell-3}(F_m)),~l_1+l_2+l_3=\ell-3;\\
   \mathbb{T}^{(\ell)}_{l}(v)&=\frac{1}{|T|}\int_{T_m}\partial_{\mathbf{n}}(v)\lambda_1^{l_1}\lambda_2^{l_2}\lambda_3^{l_3}\lambda_{4}^{l_4}\mathrm{d}\mathbf{x},~l=1,\cdots,\mathrm{dim}(P_{\ell-4}(T)),~l_1+l_2+l_3+l_4=\ell-4.
\end{align*}
\subsection{Enriched $P_4(T)$ element in 3D}  When $\ell=4$, the degrees of freedom of \eqref{degree1}--\eqref{degree4} can be equivalently 
  rewritten as follows, respectively.
\begin{equation}
\mathbb{E}^{(4)}_{m,l}(v)
    =\frac{1}{|e_m|}\int_{e_m}v\lambda_i^{l_1}\lambda_j^{l_2}\mathrm{d}s,
    m=1,\cdots,6;  \label{degree11}
\end{equation}
here $(l_1,l_2)$ is $(2,0),(1,1),(0,2)$ for $l=1,2,3$, respectively. And in \eqref{degree11}, $\lambda_i$ and $\lambda_j$ are two barycentric coordinates of $e_m;$
\begin{equation}
\mathbb{F}^{(4)}_{m,l}(v)
  =\frac{1}{|F_m|}\int_{F_m}v\lambda_i^{l_1}\lambda_j^{l_2}\lambda_k^{l_3}\mathrm{d}S,\
   m=1,\cdots,4,
    \label{degree12}
    \end{equation}
where $(l_1,l_2,l_3)$ is $(1,0,0),(0,1,0),(0,0,1)$ for $l=1,2,3,$ respectively. And in \eqref{degree12} $\lambda_i,\lambda_j$ and $\lambda_k$ are three barycentric coordinates of $F_m;$
    \begin{equation}
  \mathbb{T}_l^{(4)}(v)=\frac{1}{|T|}\int_Tv \mathrm{d}\mathbf{x},
       \label{degree13}~l=1;\\
\end{equation}
\begin{equation}
\mathbb{N}^{(4)}_{m,l}(v)
    =\frac{1}{|F_m|}
\int_{F_m}\partial_{\mathbf{n}}v\lambda_i^{l_1}\lambda_j^{l_2}\lambda_k^{l_3}\mathrm{d}S ,~m=1,\cdots,4,
\label{degree14}
\end{equation}
where $(l_1,l_2,l_3)$ is $$(2,0,0),(0,2,0),(0,0,2),(0,1,1),(1,0,1),(1,1,0)$$ for $l=1,\cdots,6$, respectively.
%In these definitions we assume that $l_t\geq 0,\lambda_i$ and $\lambda_j$
%are two barycentric coordinates of $e_m.$ $\lambda_i,\lambda_j$ and $\lambda_k$ are three barycentric coordinates of $F_m.$

The number of degrees of freedom of (\ref{degree11}) to (\ref{degree14})
is 55 which is less than dim$P_5(T)$=56. However, there are no enough linearly independent polynomials in $P_5(T)$ with respect to \eqref{degree11}
to \eqref{degree14}. In fact there are four
     nonzero functions in $P_5(T)$ which vanish for all the degrees of freedom of \eqref{degree11} to \eqref{degree14}.
To see it, define
\begin{equation}
   b_{i} = \lambda_i^2(\lambda_i^3-\frac{15}{8}\lambda_i^2
           +\frac{15}{14}\lambda_i-\frac{5}{28}),~i=1,2,3,4.
\end{equation}
Note that $b_i$ vanishes for all the degrees of freedom of (\ref{degree11})--(\ref{degree14}).  Hence, the lowest polynomial degree for
    the enrichment is six.
    
Next we present 24 functions associated to the degrees of freedom of (\ref{degree14}). The 24 basis functions consist of four $P_4$ functions which are defined in (\ref{se}) and twenty $P_6$ functions which vanish for the degrees of freedom of (\ref{degree11}) to (\ref{degree13}).  \\
Firstly, let $\ell=4$ in (\ref{se}) we get four functions as 
\begin{equation}
   \tilde{\phi}^{(N,4)}_{m,0}=\frac{1}{4}(35\lambda_m^4 - 60\lambda_m^3 + 30\lambda_m^2 -4\lambda_m),
      ~ m=1,\cdots,4 .
      \label{basen1}
\end{equation}
Next, we define the following twelve $P_6$ functions
\begin{equation}
\tilde{\phi}^{(N,4)}_{m,t}=\frac{1}{2}\lambda_t^2\lambda_m(42\lambda_m^3-56\lambda_m^2+21\lambda_m-2),1\leq m\neq t\leq 4.
\label{basen2}
\end{equation}
Lastly, the rest eight $P_6$ functions are defined in \eqref{basen3} as follows. 
\begin{equation}
\tilde{\phi}^{(N,4)}_{m,4+t}=\frac{1}{2}b_{F_t}(42\lambda_m^3-56\lambda_m^2+21\lambda_m-2),1\leq m\neq t\leq 4.
\label{basen3}
\end{equation}
Recall that $b_{F_t}$ is the cubic face bubble function with respect to face $F_t$.
In fact, there are twelve functions in \eqref{basen3}  but we only need eight of them. Note that the 28 functions in \eqref{basen1} to \eqref{basen3} vanish for the degrees of freedom of \eqref{degree11} to \eqref{degree13} and satisfy
\begin{equation}
    r_n\mathbb{N}^{(4)}_{n,l}(\tilde{\phi}^{(N,4)}_{m,0})=
  \frac{\delta_{m,n}2!l_1!l_2!l_3!}{(l_1+l_2+l_3+2)!},
             ~ 1\leq n,m\leq 4,~l=1,\cdots,6,
\label{m1}
\end{equation}
where $r_n$ is the distance between vertex $n$ and the face triangle $F_n$; 
\begin{equation}
r_n\mathbb{N}^{(4)}_{n,l}(\tilde{\phi}^{(N,4)}_{m,t})=
\frac{\delta_{m,n}2!(l_1+2\delta_{i,t})!(l_2+2\delta_{j,t})!(l_3+2\delta_{k,t})!}{(l_1+l_2+l_3+2)!},
\label{m2}
\end{equation}
where $1\leq m,n\leq 4~,1\leq t\neq m\leq 4,l=1,\cdots,6$ and $i,j,k$ are the index of three vertices on $F_m$;
\begin{equation}
r_n\mathbb{N}^{(4)}_{n,l}(\tilde{\phi}^{(N,4)}_{m,4+t})=
\frac{\delta_{m,n}2!(l_1+1-\delta_{i,t})!(l_2+1-\delta_{j,t})!(l_3+1-\delta_{k,t})!}{(l_1+l_2+l_3+2)!},
\label{m3}
\end{equation}
where $1\leq m,n\leq 4,~1\leq t\neq m\leq 4,~l=1,\cdots,6.$
%Note that in (\ref{f1}), there are 12 functions rather than eight functions. a relatively symmetrical method to choose the 8 $P_6(T)$ functions is given by \eqref{choose1}. However, the selection is not unique. \\
%Note that the selection in () is not unique. In order to use two face bubble on each face, we consider this selection.
Thus, we define the following shape function space:
\begin{equation}
	\tilde{\mathcal{P}}^+_4(T) = P_4(T)+\tilde{B}_4(T),
	\label{Vhl3}
\end{equation}
where
\begin{equation}
\tilde{B}_4(T)=\mathrm{span}\{\tilde{\phi}^{(N,4)}_{m,t},\tilde{\phi}^{(N,4)}_{1,7},\tilde{\phi}^{(N,4)}_{1,8},
    \tilde{\phi}^{(N,4)}_{2,8},\tilde{\phi}^{(N,4)}_{2,5},\tilde{\phi}^{(N,4)}_{3,5},\tilde{\phi}^{(N,4)}_{3,6},
     \tilde{\phi}^{(N,4)}_{4,6},\tilde{\phi}^{(N,4)}_{4,7} \}
\label{choose1}
\end{equation}
with $ 1\leq m\neq t\leq 4$. Note that the multi index $(m,t)$ of $\tilde{\phi}^{(4)}_{m,4+t}$ are chosen as $$(1,3),(1,4),(2,4),(2,1),(3,1),(3,2),(4,2),(4,3).$$
The 20 functions of $\tilde{B}_4(T)$ are linearly independent, which will be shown
in Theorem \ref{unisolve2} below.
Then the global finite element space is defined by
\begin{align*}
V_{h,4}=&\{v\in L^2{\Omega}|~v|_T\in \tilde{\mathcal{P}}^+_4(T),\\
&\int_e vp_2 \r ds~\text{is continuous at internal edges of }\mathcal{T}_h,\\
&\int_F vp_1\r dS ~\text{and}~\int_F \partial_{\mathbf{n}}vp_2 \r dS
\text{ are continuous on internal face triangles of } \mathcal{T}_h\\
&\int_e vp_2 \r ds=\int_F vp_1 \r dS =\int_F \partial_{\mathbf{n}}vp_2
\r dS=0\\
&\text{at boundary edges and on boundary face triangles of }\mathcal{T}_h
\}.
\end{align*}
Next, we construct the rest 31 $P_4$ functions which
 do not vanish for normal derivative moments of \eqref{degree14}. %Firstly, we consider the basis function associated with dofs (\ref{degree11}).  We can easily get the following three functions by solving the linear equations. The linear equations are obtained by applying $\lambda_i^2,\lambda_i\lambda_j,\lambda_j^2$ as dual basis functions of $\mathbb{E}^4_{m,l},l=1,2,3$ on $\tilde{\varphi}^{e,4}_{m,i}$ and satisfy

  %Here we suppose $\tilde{\varphi}^{e,4}_{m,i}$ have a such form as
  %$$% \tilde{\varphi}^{e,4}_{m,i}=b_{e_m}(a_0\lambda_i^2+a_1\lambda_i\lambda_j+a_2\lambda_j^2).
 % $$
Firstly we define the following 18 $P_4$ function as
\begin{equation}
   \ad{ \tilde{\phi}^{(E,4)}_{m,1} & =60b_{e_m}(7\lambda_i^2-6\lambda_i+1), \\
         \tilde{\phi}^{(E,4)}_{m,2} &=60b_{e_m}(21\lambda_i\lambda_j-6\lambda_i-6\lambda_j+2), \\
       \tilde{\phi}^{(E,4)}_{m,3} &=60b_{e_m}(7\lambda_j^2-6\lambda_j+1), }
      \label{basef2}
\end{equation}
where  $\lambda_i$ and $\lambda_j(i<j)$ are two barycentric coordinates of edge $e_m$, and $b_{e_m}=\lambda_i\lambda_j,m=1,\cdots,6$. We also need the following twelve $P_4$ functions:
\begin{equation}
\begin{array}{cc}
   \tilde{\phi}^{(F,4)}_{m,1}&=180b_{F_m}(7\lambda_{i}-2),\\
   \tilde{\phi}^{(F,4)}_{m,2}&=180b_{F_m}(7\lambda_{j}-2),\\
   \tilde{\phi}^{(F,4)}_{m,3}&=180b_{F_m}(7\lambda_{k}-2),
\label{basef3}
\end{array}
\end{equation}
where $\lambda_i,\lambda_j,\lambda_k$ are three barycentric coordinates of triangle face $F_m$, and  $b_{F_m}=\lambda_i\lambda_j\lambda_k,~m=1,\cdots,4.$
The last one $P_4$ function that we need is
\begin{equation}
\tilde{\phi}^{(T,4)}_{1} = 840\lambda_1\lambda_2\lambda_3\lambda_4.
\label{baset}
\end{equation}
The 31 $P_4$ functions, defined in \eqref{basef2},\eqref{basef3} and \eqref{baset}, satisfy
    \begin{equation}
    \mathbb{E}^{(4)}_{n,l}(\tilde{\phi}^{(E,4)}_{m,t})	=\delta_{n,m}\delta_{l,t},~\mathbb{F}^{(4)}_{r,l}(\tilde{\phi}^{(E,4)}_{m,t}) = 0,~\mathbb{T}^{(4)}_{1}(\tilde{\phi}^{(E,4)}_{m,t})=0,
  \label{basef41}
  \end{equation}
  where $1\leq l,t\leq 3,~1\leq r\leq 4,~1\leq m,n\leq 6;$
  \begin{equation}
    \mathbb{E}^{(4)}_{n,l}(\tilde{\phi}^{(F,4)}_{m,t})	=0,~\mathbb{F}^{(4)}_{r,l}(\tilde{\phi}^{(F,4)}_{m,t}) = \delta_{r,m}\delta_{l,t},~\mathbb{T}^{(4)}_{1}(\tilde{\phi}^{(F,4)}_{m,t})=0,
  \label{basef42}
  \end{equation}
  where $1\leq l,t\leq 3,~1\leq m,r\leq 4,~1\leq n\leq 6;$
   \begin{equation}
    \mathbb{E}^{(4)}_{n,l}(\tilde{\phi}^{(T,4)}_{1})	=0,~\mathbb{F}^{(4)}_{r,l}(\tilde{\phi}^{(T,4)}_{1}) = 0,~\mathbb{T}^{(4)}_{1}(\tilde{\phi}^{(T,4)}_{1})=1,
  \label{basef43}
\end{equation}
where $1\leq l\leq 3,~1\leq r\leq 4,~1\leq n\leq 6.$
%Here the dual basis functions of $\mathbb{F}^4_{m,l},l=1,2,3$ are $\lambda_i,\lambda_j,\lambda_k$ .
%  Moreover $\tilde{\varphi}^{F,4}_{m,l}$ satisfies
%\a{
%     \mathbb{F}^4_{n,j}(\tilde{\varphi}^{F,4}_{m,l})=\delta_{n,m}\delta_{j,l}.}
 %The last basis function is
%corresponding to the degree of freedom (\ref{degree13}).
\begin{myTheo} $\tilde{\mathcal{P}}^+_4(T)$ is unisolvent for the
                 degrees of freedom of (\ref{degree11}) to (\ref{degree14}).
\label{unisolve2}
\end{myTheo}

{\begin{proof}
%We now show that these functions from \eqref{basen1} to \eqref{basen3} and \eqref{basef2} to \eqref{baset} form a basis of $P^+_4(T)$. We only need to show they are linearly independent.
%Let $u_h\in \mathcal{P}^+_4(T)$ such that the 55 degrees of freedom of $u_h$ in (\ref{degree11}) to
%    (\ref{degree14}) have zero value.
%We need to show $u_h=0$.
Assume that
\a{ u_h=\sum_{m=1}^{6}\sum_{l=1}^{3}a_{m,l}\tilde{\phi}^{(E,4)}_{m,l}+\sum_{m=1}^{4}\sum_{l=1 }^{4}b_{m,l}\tilde{\phi}^{(F,4)}_{m,l}+c\tilde{\phi_1}^{(T,4)}+\sum_{m=1}^{4}\sum_{t\in\mathcal{S}_m}d_{m,t}\tilde{\phi}^{(N,4)}_{m,t}£¬}
for parameters $a_{m,l},b_{m,l},c$ and $d_{m,l}.$
Here $\mathcal{S}_m=\{t\in \mathbb{Z}|0\leq t\leq 4,t\neq m\}\cup\{j+4,k+4\}$, where $j$ is the index of the second vertex and $k$ is the index of the third on $F_m.$
%Here, for convenience, we note $$\tilde{\varphi}^{N,(4)}_{m,0},\tilde{\varphi}^{N,(4)}_{m,t},\tilde{\varphi}^{N,(4)}_{1,7},\tilde{\varphi}^{N,(4)}_{1,8},
%\tilde{\varphi}^{N,(4)}_{2,8},\tilde{\varphi}^{N,(4)}_{2,5},\tilde{\varphi}^{N,(4)}_{3,5},\tilde{\varphi}^{N,(4)}_{3,6},
%\tilde{\varphi}^{N,(4)}_{4,6},\tilde{\varphi}^{N,(4)}_{4,7},1\leq m\neq t\leq 4$$ which are defined in \eqref{basen1},\eqref{basen2} and \eqref{basen3} as $\tilde{\varphi}^{N,(4)}_t,t=1,\cdots,24.$
We only need to show if $u_h$ vanishes for the degrees of freedom of \eqref{degree11} to \eqref{degree14} then $u_h$ vanishes.
    From \eqref{basef41} to \eqref{basef43},
    \a{ \ad{&\hbox{applying degrees of freedom of \eqref{degree11}--\eqref{degree13} to $u_h$} &&\Rightarrow a_{m,l}=0,b_{m,l}=0,c=0.
     } }
Next we show that the remaining 24 coefficients $d_{m,t},t\in \mathcal{S}_m(m=1,\cdots,4)$ are also zero.
An application of functionals $r_m\mathbb{N}^{(4)}_{m,l}(\cdot)(m=1,\cdots,4,~l=1,\cdots,6)$ to $u_h$,
%    where the dual basis functions of $\mathbb{N}^{(4)}_{m,l},l = 1,\cdots,6$ are
%$$   \lambda_{i}^2,\lambda_j^2,\lambda_{k}^2,
%    \lambda_i\lambda_{j},\lambda_{i}\lambda_{k},\lambda_j\lambda_{k}.
%$$
  yields the matrix
  $\operatorname{diag}\{A,A,A,A\}$ after exchanging rows with
\a{ A = \left(
\begin{array}{cccccc}
1/6     &       1/15     &      1/90    &       1/90      &     1/60     &      1/60 \\
1/6     &       1/90     &      1/15    &       1/90      &     1/180    &      1/60  \\
1/6     &       1/90     &      1/90    &       1/15      &     1/60     &      1/180  \\
1/12    &       1/60     &      1/60    &       1/180     &     1/180    &      1/180 \\
1/12    &       1/60     &      1/180   &       1/60      &     1/90     &      1/180 \\
1/12    &       1/180    &      1/60    &       1/60      &     1/180    &      1/90
\end{array}
\right). } The fact $\det(A)\ne 0$ completes the proof.
\end{proof}

\begin{myTheo}
The finite element solution $u_h\in V_{h,4}$ in (\ref{dproblem}) satisfies
$$
|u-u_h|_{2,h}\leq Ch^3|u|_5,
$$
where $u\in H^5(\Omega)\cap H_0^2(\Omega)$ and $C$ is independent of $h$.
\end{myTheo}

   We will present all the basis functions in appendix.

\subsection{Enriched $P_5(T)$ in 3D} The degrees of freedom for this case are
\begin{equation}
   \mathbb{E}^{(5)}_{m,l}(v)
    =\frac{1}{|e_m|}\int_{e_m}v\lambda_i^{l_1}\lambda_j^{l_2}\mathrm{d}s,~ m=1,\cdots,6,
       \label{degree21}
\end{equation}
where $(l_1,l_2)$ is $(3,0),(0,3),(2,1),(1,2)$ for $l=1,\cdots,4$, respectively;
\begin{equation}
\mathbb{F}^{(5)}_{m,l}(v)
      =\frac{1}{|F_m|}\int_{F_m}v\lambda_i^{l_1}\lambda_j^{l_2}\lambda_k^{l_3}\mathrm{d}S,~
       m=1,\cdots,4;
      \label{degree22}
\end{equation}
here $(l_1,l_2,l_3)$ is
$$(2,0,0),(0,2,0),(0,0,2),(0,1,1),(1,0,1),(1,1,0)$$ for $l=1,\cdots,6$, respectively;
\begin{equation}
\mathbb{T}^{(5)}_l(v)
     =\frac{1}{|T|}\int_Tv\lambda_l \mathrm{d}\mathbf{x},~ l=1,\cdots,4;
     \label{degree23}
\end{equation}
\begin{equation}
\mathbb{N}^{(5)}_{m,l}(v)
    =\frac{1}{|F_m|}
     \int_{F_m}\partial_{\mathbf{n}}v\lambda_i^{l_1}\lambda_j^{l_2}\lambda_k^{l_3}
                 \mathrm{d}S,~m=1,\cdots,4,
    \label{degree24}
\end{equation}
where the multi index $(l_1,l_2,l_3)$ is
\begin{equation} 
(3,0,0),(0,3,0),(0,0,3),(2,1,0),(2,0,1),(0,2,1),(1,2,0),(1,0,2),(0,1,2),(1,1,1)
\label{index5}
\end{equation}
for $l=1,\cdots,10,$ respectively.
%The local shape function space consists of $P_5+40P_7$.
  We need $P_7$ polynomials because the number of degrees of freedom of
    (\ref{degree21}) to (\ref{degree24}) is 92,
      which is bigger than $\text{dim}P_6(T)=84$. Next we present 40 functions associated with degrees of freedom of \eqref{degree24}, which vanish for degrees of freedom of \eqref{degree21} to \eqref{degree23}. Firstly, let $\ell=5$ in \eqref{se} we get four functions
\begin{equation}
  \tilde{\phi}^{(N,5)}_{m,0}=-\frac{1}{5}\lambda_m(126\lambda_m^4-280\lambda_m^3+210
    \lambda_m^2-60\lambda_m+5),~ m=1,\cdots,4.
     \label{basen41}
\end{equation}
Secondly we define the 24 $P_7$ functions as follows
\begin{equation}
\tilde{\phi}^{(N,5)}_{m,t}=-\lambda_t^2\lambda_m(66\lambda_m^4-120\lambda_m^3
+72\lambda_m^2-16\lambda_m+1),~1\leq m\neq t\leq 4.
\label{basen42}
\end{equation}
\begin{align}
\tilde{\phi}^{(N,5)}_{m,4+t}=&-\frac{1}{4}\lambda_t^2\lambda_m(
99\lambda_m^4-165\lambda_m^3\lambda_t-135\lambda_m^3+180\lambda_m^2\lambda_t
  +54\lambda_m^2\label{basen43}\\&-54\lambda_t\lambda_m-6\lambda_m+4\lambda_t),~1\leq m\neq t\leq 4\notag.
\end{align}
Define the following twelve $P_7$ functions
\begin{equation}
\tilde{\phi}^{(N,5)}_{m,8+t}=-b_{F_t}(66\lambda_m^4-120\lambda_m^3
+72\lambda_m^2-16\lambda_m+1),~1\leq m\neq t\leq 4.
\label{basen44}
\end{equation}
 The eight functions which we need, can be gotten by choosing the multi index $(m,t)$ as $$(1,3),(1,4),(2,4),(2,1),(3,1),(3,2),(4,2),(4,3).$$ And the last four functions that we need are
\begin{equation}
\tilde{\phi}^{(N,5)}_{m,13}=\frac{1}{4}\lambda_ib_{F_k}(165\lambda_m^3-180\lambda_m^2+54\lambda_m-4),
\label{basen45}~m=1,\cdots,4,	
\end{equation}
where $i,k$ are the index of the first and the third vertices on $F_m$, respectively.
 Thus, we get all 40 functions associated with degrees of freedom of \eqref{degree24}. Similarly, define the shape function space as
\an{\label{Vhl4} \tilde{\mathcal{P}}^+_5(T)&=P_5(T)+ \tilde{B}_5(T) }
 \an{\notag\text{with } \tilde{B}_5(T)&= \operatorname{span}
    \{\tilde{\phi}^{(N,5)}_{m,t},\tilde{\phi}^{(N,5)}_{m,4+t},\tilde{\phi}^{(N,5)}_{1,11},\tilde{\phi}^{(N,5)}_{1,12},\tilde{\phi}^{(N,5)}_{2,12},
 \tilde{\phi}^{(N,5)}_{2,9},\tilde{\phi}^{(N,5)}_{3,9},  \\&\notag\qquad
    \tilde{\phi}^{(N,5)}_{3,10},\tilde{\phi}^{(N,5)}_{4,10},
     \tilde{\phi}^{(N,5)}_{4,11},\tilde{\phi}^{(N,5)}_{1,13},\tilde{\phi}^{(N,5)}_{2,13},
      \tilde{\phi}^{(N,5)}_{3,13},\tilde{\phi}^{(N,5)}_{4,13}\},~ 1\leq m\neq t\leq 4.
       }
       The 36 functions of $\tilde{B}_5(T)$ are linearly independent, which will be shown
       in Theorem \ref{unisolve3} below.
Then the global finite element space is defined by
\begin{align*}
V_{h,5}=&\{v\in L^2{\Omega}|~ v|_T\in \tilde{\mathcal{P}}^+_5(T),\\
&\int_e vp_3\mathrm{d}s~\text{ is continuous at internal edges of }\mathcal{T}_h,\\
&\int_F vp_2 \r dS ~\text{and}~\int_F \partial_{\mathbf{n}}vp_3\mathrm{d}S
\text{ are continuous on internal face triangles of } \mathcal{T}_h\\
&\int_e vp_3\mathrm{d}s=\int_F vp_2 \r dS=\int_F \partial_{\mathbf{n}}v p_3\mathrm{d}S=0\\
&\text{at boundary edges and on boundary face triangles of }\mathcal{T}_h
\}.
\end{align*} % where $\mathcal{P}^+_5(T)$ is defined in \eqref{Vhl4}.       
Next we define 52 $P_5$ functions which do not vanish for
          degrees of freedom of (\ref{degree24}).
 First of all, the following 24 $P_5$ functions with an edge bubble $b_{e_m}=\lambda_i\lambda_j$ are needed,
\begin{align*}
\tilde{\phi}^{(E,5)}_{m,1}&=60b_{e_m}
(42\lambda_i^3-56\lambda_i^2+21\lambda_i-2),\\
\tilde{\phi}^{(E,5)}_{m,2}&=60b_{e_m}
(42\lambda_j^3-56\lambda_j^2+21\lambda_j-2),\\
\tilde{\phi}^{(E,5)}_{m,3}&=60b_{e_m}
(252\lambda_i^2\lambda_j-168\lambda_i\lambda_j-56\lambda_i^2+
42\lambda_i+21\lambda_j-6),\\
\tilde{\phi}^{(E,5)}_{m,4}&=60b_{e_m}
(252\lambda_j^2\lambda_i-168\lambda_i\lambda_j-56\lambda_j^2+
42\lambda_j+21\lambda_i-6),
\end{align*}
where $\lambda_i$ and $\lambda_j(i<j)$ are two barycentric coordinates of edge $e_m,~m=1,\cdots,6.$
%Here the dual basis functions of $\mathbb{E}_{m,l}^5,l=1,\cdots,6$ are
%  $\lambda_i^3,\lambda_j^3,\lambda_i^2\lambda_j,\lambda_j^2\lambda_i$,
%    where edge $e_m$ has vertexes $i$ and $j$,  $i<j$.
Next, we need the following twenty-four $P_5$ basis functions with a face bubble function $b_{F_m}=\lambda_i\lambda_j\lambda_k,$
\a{
     \tilde{\phi}^{(F,5)}_{m,1}&=1260b_{F_m}f_1(\lambda_i),&
     \tilde{\phi}^{(F,5)}_{m,2}&= 1260b_{F_m}f_1(\lambda_j),\\
     \tilde{\phi}^{(F,5)}_{m,3}&= 1260b_{F_m}f_1(\lambda_k),&
     \tilde{\phi}^{(F,5)}_{m,4}&= 2520b_{F_m}f_2(\lambda_i,\lambda_j),\\
     \tilde{\phi}^{(F,5)}_{m,5}&= 2520b_{F_m}f_2(\lambda_i,\lambda_k),&
     \tilde{\phi}^{(F,5)}_{m,6}&= 2520b_{F_m}f_2(\lambda_j,\lambda_k), }
where $\lambda_i,\lambda_j,\lambda_k$ are three barycentric coordinates of face $F_m,$ and $f_1(\lambda_l)=12\lambda_l^2-8\lambda_l+1,~f_2(\lambda_{l_1},\lambda_{l_2})=18\lambda_{l_1}
      \lambda_{l_2} -4\lambda_{l_1}-4\lambda_{l_2}+1,m=1,\dots,4.$
%Here the six dual basis functions of $\mathbb{F}^5_{m,l},l=1,\cdots,6$ are  $\lambda_i^2,\lambda_j^2, \lambda_k^2,
 %          \lambda_i\lambda_j,\lambda_i\lambda_k,\lambda_j\lambda_k$.
The last four $P_5$ functions that we need are
\begin{equation*}
\tilde{\phi}^{(T,5)}_{m}=3360\lambda_1\lambda_2\lambda_3\lambda_4(9\lambda_m-2),~ m=1,\cdots,4.
\end{equation*}
And the 52 $P_5$ functions satisfy
$$
\mathbb{E}^{(5)}_{n,l}(\tilde{\phi}^{(E,5)}_{m,t})=\delta_{m,n}\delta_{l,t},~\mathbb{F}^{(5)}_{l,n}(\tilde{\phi}^{(E,5)}_{m,t})=0,~\mathbb{T}^{(5)}_{l}(\tilde{\phi}^{(E,5)}_{m,t})=0,
$$
where $~1\leq m,n\leq 6,~
1\leq l,t\leq 4;$
$$
\mathbb{E}^{(5)}_{n,l}(\tilde{\phi}^{(F,5)}_{m,t})=0,~\mathbb{F}^{(5)}_{l,n}(\tilde{\phi}^{(F,5)}_{m,t})=\delta_{m,l}\delta_{n,t},~\mathbb{T}^{(5)}_{l}(\tilde{\phi}^{(F,5)}_{m,t})=0,
$$
where $1\leq n,t\leq 6,~1\leq m,l\leq 4;$
$$
\mathbb{E}^{(5)}_{n,l}(\tilde{\phi}^{(T,5)}_{t})=0,~\mathbb{F}^{(5)}_{l,n}(\tilde{\phi}^{(T,5)}_{t})=0,~\mathbb{T}^{(5)}_{l}(\tilde{\phi}^{(T,5)}_{t})=\delta_{l,t},
$$
where $1\leq n\leq 6,~1\leq l,t\leq 4.$
%Here the dual basis functions of $\mathbb{T}^5_l,l=1,\cdots,4$ are $\lambda_i,i=1,\cdots,4$.

\begin{myTheo}
$\tilde{\mathcal{P}}^+_5(T)$ is unisolved for degrees of freedom of (\ref{degree21}) to (\ref{degree24}).
\label{unisolve3}
\end{myTheo}
\begin{proof}
Let $u_h\in \tilde{\mathcal{P}}^+_5(T)$ vanish for all 92 degrees of freedom of (\ref{degree21}) to
(\ref{degree24}). Then we only need to show $u_h=0$. Assume that
\begin{equation} \label{eq5}
   u_h=\sum_{m=1}^{6}\sum_{l=1}^{4}a_{m,l}\tilde{\phi}^{(E,5)}_{m,l}+\sum_{m=1}^{4}\sum_{l=1}^{6}b_{m,l}
    \tilde{\phi}^{(F,5)}_{m,l}+\sum_{m=1}^{4}c_{m}\tilde{\phi}^{(T,5)}_{m}
     +\sum_{m=1}^{4}\sum_{l\in \mathcal{S}_m}d_{m,l}\tilde{\phi}^{(N,5)}_{m,l}=0,
\end{equation}
for parameters $a_{m,l},b_{m,l},c_m$ and $d_{m,l}$. % Here we note
%$$
%\tilde{\varphi}^{N,(5)}_{m,t},\tilde{\varphi}^{N,(5)}_{m,4+t},\tilde{\varphi}^{N,(5)}_{m,0},\tilde{\varphi}^{N,(5)}_{1,11},\tilde{\varphi}^{N,(5)}_{1,12},\tilde{\varphi}^{N,(5)}_{2,12},
%\tilde{\varphi}^{N,(5)}_{2,9},\tilde{\varphi}^{N,(5)}_{3,9},\tilde{\varphi}^{N,(5)}_{3,10},\tilde{\varphi}^{(5)}_{4,10},
%	\tilde{\varphi}^{N,(5)}_{4,11},\tilde{\varphi}^{N,(5)}_{1,13},\tilde{\varphi}^{N,(5)}_{2,13},
%	\tilde{\varphi}^{N,(5)}_{3,13},\tilde{\varphi}^{N,(5)}_{4,13},
%$$
%where $1\leq m\neq t\leq 4$, defined in \eqref{basen41} to \eqref{basen45}, as
%$\tilde{\varphi}^{N,(5)}_{l},l=1,\cdots,40.$
 Here $$\mathcal{S}_m=\{t\in \mathbb{Z}|0\leq t\leq 8,t\neq m,t\neq m+4\}\cup\{j+8,k+8,13\},$$
 where $j,k$ are the index of the second and the third vertices on $F_m$, respectively.
 Sequentially,
 \a{ \ad{&\hbox{applying degrees of freedom of \eqref{degree21}--\eqref{degree23} to $u_h$} &&\Rightarrow a_{m,l}=0,b_{m,l}=0,c_m=0.
 		} }
Then we apply the functional $r_m\mathbb{N}^{(5)}_{m,l}(\cdot)(m=1,\cdots,4,~l=1,\cdots,10)$ to $u_h$,
to get the following matrix for the linear system \eqref{eq5} after exchanging rows
\begin{equation}
\left(
\begin{array}{cccccccc}
A&\mathbf{a}&\mathbf{0}&c_1\mathbf{e_1} &\mathbf{0}&c_2\mathbf{e_2}&\mathbf{0}&\mathbf{0}\\
\mathbf{0}&\mathbf{0}& A&\mathbf{a}&\mathbf{0} &c_1\mathbf{e_1}&\mathbf{0}&c_2\mathbf{e_2} \\
\mathbf{0}&c_2\mathbf{e_2} & \mathbf{0}&\mathbf{0} & A&\mathbf{a}&\mathbf{0} &c_1\mathbf{e_1}\\
\mathbf{0}&c_1\mathbf{e_1} &\mathbf{0}&c_2\mathbf{e_2}&\mathbf{0} &\mathbf{0} & A&\mathbf{a}
\end{array}
\right),
\label{matrix5}
\end{equation}
where $c_1=-c_2=\frac{1}{1680}$, $\mathbf{e}_t$ is the $t$-th canonical basis vector of $\mathbb{R}^{10}$.
$$
[A,\mathbf{a}]=\left(\begin{array}{cccccccccc}
       1/10   &        1/21     &      1/210    &      1/210    &      1/28     &      1/560    &      1/560    &      1/105    &      1/105   &       1/168  \\ 
       1/10   &        1/210    &      1/21     &      1/210    &      1/560    &      1/28     &      1/560    &      1/420    &      1/105   &       1/420  \\
       1/10   &        1/210    &      1/210    &      1/21     &      1/560    &      1/560    &      1/28     &      1/105    &      1/420   &       1/1680 \\ 
       1/30   &        1/105    &      1/210    &      1/630    &      1/168    &      1/420    &      1/1680   &      1/420    &      1/210   &       1/420  \\ 
       1/30   &        1/105    &      1/630    &      1/210    &      1/168    &      1/1680   &      1/420    &      1/210    &      1/420   &       1/840  \\ 
       1/30   &        1/630    &      1/105    &      1/210    &      1/1680   &      1/168    &      1/420    &      1/630    &      1/420   &       1/1680 \\ 
       1/30   &        1/210    &      1/105    &      1/630    &      1/420    &      1/168    &      1/1680   &      1/630    &      1/210   &       1/560  \\ 
       1/30   &        1/210    &      1/630    &      1/105    &      1/420    &      1/1680   &      1/168    &      1/210    &      1/630   &       1/1680 \\ 
       1/30   &        1/630    &      1/210    &      1/105    &      1/1680   &      1/420    &      1/168    &      1/420    &      1/630   &       1/2520 \\ 
       1/60   &        1/420    &      1/420    &      1/420    &      1/840    &      1/840    &      1/840    &      1/630    &      1/630   &       1/1680    
\end{array}
\right)
$$
 One can check the matrix in (\ref{matrix5}) is invertable which completes the proof.
\end{proof}

\begin{myTheo}
	The finite element solution $u_h\in V_{h,5}$ in (\ref{dproblem}) satisfies
	$$
	|u-u_h|_{2,h}\leq Ch^4|u|_6,
	$$
	where $u\in H^6(\Omega)\cap H_0^2(\Omega)$ and $C$ is independent of $h$.
\end{myTheo}

\section{Numerical tests}
Let the  domain of the boundary value problem~\eqref{problem} be the unit cubic
$\Omega=(0,1)^3$.
The exact solution is
\begin{equation} \label{s1} u(x,y,z)=2^{10} (x-x^2)^2 (y-y^2)^2 (z-z^2)^2. \end{equation}
We choose a family of uniform grids, shown in Figure~\ref{grid}, for all
tests.

\begin{figure}[htb] \setlength\unitlength{1pt}
	\begin{center}  \begin{picture}(380,100)(0,0)
		\put(0,0){\begin{picture}(100,100)(0,0)
			\multiput(0,0)(80,0){2}{\line(0,1){80}}     \multiput(0,0)(0,80){2}{\line(1,0){80}}
			\multiput(80,0)(0,80){2}{\line(2,1){40}}    \multiput(0,80)(80,0){2}{\line(2,1){40}}
			\multiput(80,0)(40,20){2}{\line(0,1){80}}    \multiput(0,80)(40,20){2}{\line(1,0){80}}
			\multiput(0,0)(40,0){1}{\line(1,1){80}}
			\multiput(80,0)(0,80){1}{\line(2,5){40}}    \multiput(0,80)(80,0){1}{\line(6,1){120}}
			\end{picture} }
		\put(130,0){\begin{picture}(100,100)(0,0)
			\multiput(0,0)(40,0){3}{\line(0,1){80}}     \multiput(0,0)(0,40){3}{\line(1,0){80}}
			\multiput(80,0)(0,40){3}{\line(2,1){20}}    \multiput(0,80)(40,0){3}{\line(2,1){20}}
			\multiput(100,10)(0,40){3}{\line(2,1){20}}    \multiput(20,90)(40,0){3}{\line(2,1){20}}
			\multiput(80,0)(20,10){3}{\line(0,1){80}}    \multiput(0,80)(20,10){3}{\line(1,0){80}}
			\multiput(0,0)(40,0){1}{\line(1,1){80}}     \multiput(40,0)(-40,40){2}{\line(1,1){40}}
			
			\multiput(80,0)(0,80){1}{\line(2,5){40}}    \multiput(0,80)(80,0){1}{\line(6,1){120}}
			\multiput(100,10)(-20,30){2}{\line(2,5){20}} \multiput(40,80)(-20,10){2}{\line(6,1){60}}
			\end{picture} }
		\put(260,0){\begin{picture}(100,100)(0,0)
			\multiput(0,0)(20,0){4}{\multiput(0,0)(0,20){4}{\line(1,1){20}}}
			\multiput(0,0)(20,0){4}{\multiput(0,0)(0,20){5}{\line(1,0){20}}}
			\multiput(0,0)(20,0){5}{\multiput(0,0)(0,20){4}{\line(0,1){20}}}
			
			\multiput(0,80)(20,0){4}{\multiput(0,0)(10,5){4}{\line(6,1){30}}}
			\multiput(0,80)(20,0){5}{\multiput(0,0)(10,5){4}{\line(2,1){10}}}
			\multiput(0,80)(20,0){4}{\multiput(0,0)(10,5){5}{\line(1,0){20}}}
			
			\multiput(80,0)(0,20){5}{\multiput(0,0)(10,5){4}{\line(2,1){10}}}
			\multiput(80,0)(0,20){4}{\multiput(0,0)(10,5){4}{\line(2,5){10}}}
			\multiput(80,0)(0,20){4}{\multiput(0,0)(10,5){5}{\line(0,1){20}}}
			
			\end{picture} }
	\end{picture} \end{center}
	\caption{ \label{grid} The first three grids on the unit cube domain $\Omega$. }
\end{figure}
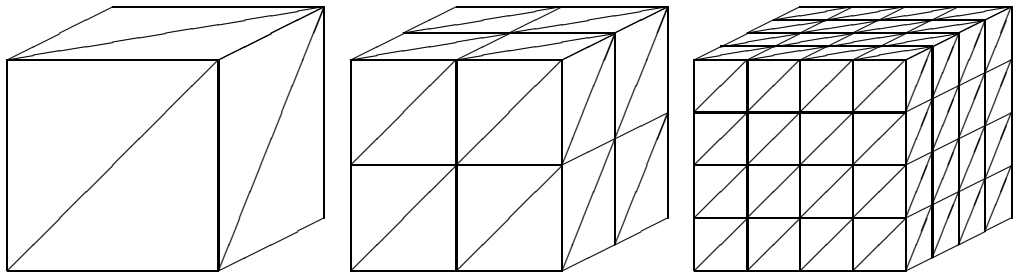

We first solve the biharmonic problem \eqref{problem} with the exact solution \eqref{s1} by the
$P_3$ finite element method \eqref{Vhl} (with $\ell=3$), i.e., the $P_3+8P_7$ element
($P_3$ polynomials plus eight $P_7$  polynomials on each tetrahedron).
We can see,  from Table \ref{t1}, that the numerical solution converges at order 2, 3 and 4 in
$H^2$-norm, $H^1$-norm and $L^2$-norm, respectively.

\begin{table}[h!]
	\caption{\label{t1} The error
		and the order of convergence, by the  $P_3+8P_7$   finite element \eqref{Vhl3} (with $\ell=3$).  }
	\begin{center}  \begin{tabular}{c|rr|rr|rr} \hline
			grid  & $ \|u-u_h\|_{0}$ &$h^n$ &$ |u-u_h|_{1}$ & $h^n$ & $  |u-u_h|_{2}$ & $h^n$   \\
			\hline
			1&     0.0506877&0.0&     0.3979739&0.0&     3.8919205&0.0\\
			2&     0.0367926&0.5&     0.2479380&0.7&     3.7328475&0.1\\
			3&     0.0053230&2.8&     0.0424815&2.5&     1.3958787&1.4\\
			4&     0.0005108&3.4&     0.0049778&3.1&     0.3761410&1.9\\
			5&     0.0000398&3.7&     0.0005033&3.3&     0.0923416&2.0\\
			\hline
		\end{tabular}\end{center} \end{table}

 We next solve the biharmonic problem \eqref{problem} with the exact solution \eqref{s1} by the
		$P_4+20P_8$ finite element method \eqref{Vhl} (with $\ell=4$), i.e.,
		the full $P_4$ polynomials plus twenty $P_8$ polynomials on each tetrahedron.
		In Table \ref{t2}, we list the orders of convergence of the numerical solutions,
		which are  3, 4 and 5 in
		$H^2$-norm, $H^1$-norm and $L^2$-norm, respectively.

		\begin{table}[h!]
			\caption{\label{t2} The error
				and the order of convergence, by the  $P_4+20P_8$   finite element \eqref{Vhl} (with $\ell=4$).  }
			\begin{center}  \begin{tabular}{c|rr|rr|rr} \hline
					grid  & $ \|u-u_h\|_{0}$ &$h^n$ &$ |u-u_h|_{1}$ & $h^n$ & $  |u-u_h|_{2}$ & $h^n$   \\
					\hline
					1&     0.0510863&0.0&     0.3390932&0.0&     4.3623945&0.0\\
					2&     0.0105988&2.3&     0.0802464&2.1&     1.7878478&1.3\\
					3&     0.0004250&4.6&     0.0057934&3.8&     0.3314596&2.4\\
					4&     0.0000079&5.8&     0.0003274&4.1&     0.0476001&2.8\\
					\hline
				\end{tabular}\end{center} \end{table}
				
 Lastly, we apply the low-order $ V_{h,4}$ \eqref{Vhl3} method to solving the 3D biharmoic equation.
 Here the full $P_4$  polynomial space is enriched by twenty $P_6$ polynomials, on each tetrahedron.
 We call it the $P_4+20P_6$ element.
 Due to a better condition number, this method is more stable than the above $P_4+20P_8$ element.
 From Table \ref{t3}, we can see the $P_4+20P_6$ element converges also at
	 order 3, 4, and 5 in
 $H^2$-norm, $H^1$-norm and $L^2$-norm, respectively.

				\begin{table}[h!]
					\caption{\label{t3} The error
						and the order of convergence, by the  $P_4+20P_6$   finite element \eqref{Vhl3}.  }
					\begin{center}  \begin{tabular}{c|rr|rr|rr} \hline
							grid  & $ \|u-u_h\|_{0}$ &$h^n$ &$ |u-u_h|_{1}$ & $h^n$ & $  |u-u_h|_{2}$ & $h^n$   \\
							\hline
							1&     0.0249382&0.0&     0.3040140&0.0&     5.9311094&0.0\\
							2&     0.0041104&2.6&     0.0645440&2.2&     2.0850503&1.5\\
							3&     0.0001847&4.5&     0.0061410&3.4&     0.3909556&2.4\\
							4&     0.0000052&5.1&     0.0004154&3.9&     0.0531467&2.9\\
							\hline
		\end{tabular}\end{center} \end{table}

\section{Appendix}
All basis functions in $P_4+20P_6$ element and $P_5+36P_7$ element will be presented in this section. Let three vertices index $(i,j,k)$ on $F_m$ be
$
(2,3,4),(3,4,1),(4,1,2),(1,2,3)
$
for $m=1,\cdots,4$, respectively. And let $\tau_m(l),1\leq m\neq l\leq 4$ be the local index of vertex $l$ on $m$-th face, for example $\tau_{3}(1)=2,\tau_{4}(1)=3.$ 
\subsection{Basis functions of the $P_4$+20$P_6$ element}
The shape function space of the $P_4$+20$P_6$ element can be equivalently rewritten as
$$
\tilde{\mathcal{P}}^+_4(T)=P_4(T)+\tilde{B}_4,
$$
\begin{align*} 
\tilde{B}_4=&\mathrm{span}\{
\lambda_t^2\lambda_m^3(3\lambda_m-4),b_{F_3}\lambda_1^2(3\lambda_1-4),b_{F_4}\lambda_1^2(3\lambda_1-4),b_{F_4}\lambda_2^2(3\lambda_2-4),b_{F_1}\lambda_2^2(3\lambda_2-4)\\&b_{F_1}\lambda_3^2(3\lambda_3-4),b_{F_2}\lambda_3^2(3\lambda_3-4),b_{F_2}\lambda_4^2(3\lambda_4-4),b_{F_3}\lambda_4^2(3\lambda_4-4)
\},~1\leq m\neq t\leq 4.
\end{align*}
Firstly, we show the following 24 basis functions dual to the degrees of freedom of \eqref{degree14}.
\begin{align*}
\phi^{(N,4)}_{m,1}&=6r_m(\tilde{\phi}^{(N,4)}_{m,0}+5\tilde{\phi}^{(N,4)}_{m,i}-10\tilde{\phi}^{(N,4)}_{m,j+4}
-10\tilde{\phi}^{(4)}_{m,k+4}),\\
\phi^{(N,4)}_{m,2}&=6r_m(-4\tilde{\phi}^{(N,4)}_{m,0}+5\tilde{\phi}^{(N,4)}_{m,i}+10\tilde{\phi}^{(N,4)}_{m,j}+5\tilde{\phi}^{(N,4)}_{m,k}
+10\tilde{\phi}^{(N,4)}_{m,j+4}),\\
\phi^{(N,4)}_{m,3}&=6r_m(-4\tilde{\phi}^{(N,4)}_{m,0}+5\tilde{\phi}^{(N,4)}_{m,i}+
5\tilde{\phi}^{(N,4)}_{m,j}+10\tilde{\phi}^{(N,4)}_{m,k}+10\tilde{\phi}^{(N,4)}_{m,k+4}),\\
\phi^{(N,4)}_{m,4}&=12r_m(11\tilde{\phi}^{(N,4)}_{m,0}-10\tilde{\phi}^{(N,4)}_{m,i}-15\tilde{\phi}^{(N,4)}_{m,j}-15\tilde{\phi}^{(N,4)}_{m,k}-25\tilde{\phi}^{(N,4)}_{m,j+4}-25\tilde{\phi}^{(N,4)}_{m,k+4}),
\\
\phi^{(N,4)}_{m,5}&=6r_m(-3\tilde{\phi}^{(N,4)}_{m,0}-5\tilde{\phi}^{(N,4)}_{m,i}+5\tilde{\phi}^{(N,4)}_{m,j}-5\tilde{\phi}^{(N,4)}_{m,k}+50\tilde{\phi}^{(N,4)}_{m,j+4}),\\
\phi^{(N,4)}_{m,6}&=6r_m(-3\tilde{\phi}^{(N,4)}_{m,0}-5\tilde{\phi}^{(N,4)}_{m,i}-
5\tilde{\phi}^{(N,4)}_{m,j}+5\tilde{\phi}^{(N,4)}_{m,k}+50\tilde{\phi}^{(N,4)}_{m,k+4}),
\end{align*}
where $m=1,\cdots,4$. Those 24 functions vanish for degrees of freedom of \eqref{degree11} to \eqref{degree13} and satisfy
\begin{align*}
\mathbb{N}^{(4)}_{m,l}(\phi_{n,t}^{(N,4)})=\delta_{m,n}\delta_{l,t},~1\leq m,n\leq 4,~1\leq l,t\leq 6.
\end{align*}
Next, we present the following basis function dual to the degree of freedom of \eqref{degree13}.
\begin{equation}
\phi^{(T,4)}_1=\tilde{\phi}^{(T,4)}_1+\sum_{m=1}^{4}\sum_{l=1}^{3}
\frac{2}{r_m}\phi^{(N,4)}_{m,l}+\sum_{m=1}^{4}\sum_{l=4}^{6}\frac{4}{3r_m}\phi^{(N,4)}_{m,l}.
\end{equation}
Here $\phi_1^{T,4}$ vanishes for degrees of freedom of \eqref{degree11},\eqref{degree12} and \eqref{degree14} and satisfies
$$
\mathbb{T}^{(4)}_1(\phi^{(T,4)}_1)=1.
$$

Then, 12 basis functions dual to degrees of freedom of \eqref{degree12} are shown as follows,
\a{
  \phi^{(F,4)}_{m,1} & = \tilde{\phi}^{(F,4)}_{m,1} -
  \frac{1}{r_i}(6\phi^{(N,4)}_{i,1}+6\phi^{(N,4)}_{i,2}+2\phi^{(N,4)}_{i,3}
    +2\phi^{(N,4)}_{i,4}+2\phi^{(N,4)}_{i,5}+4\phi^{(N,4)}_{i,6})\\ &\qquad
    +   \frac{1}{r_j}(6\phi^{(N,4)}_{j,3}+\phi^{(N,4)}_{j,4}+2\phi^{(N,4)}_{j,5})
        +\frac{1}{r_k}(6\phi^{(N,4)}_{k,2}+2\phi^{(N,4)}_{k,4}+\phi^{(N,4)}_{k,6})\\ &\qquad
   -  (\mathbf{n}_m)^T((\frac{2\mathbf{n}_i}{r_i}+\frac{6\mathbf{n}_m}{r_m})\phi^{(N,4)}_{m,1}+
    (\frac{\mathbf{n}_k}{r_k}+\frac{2\mathbf{n}_m}{r_m})\phi^{(N,4)}_{m,5}+  (\frac{\mathbf{n}_j}{r_j}+
           \frac{2\mathbf{n}_m}{r_m})\phi^{(N,4)}_{m,6}), \\
  \phi^{(F,4)}_{m,2} & = \tilde{\phi}^{(F,4)}_{m,2} -
  \frac{1}{r_j}(6\phi^{(N,4)}_{j,1}+2\phi^{(N,4)}_{j,2}+6\phi^{(N,4)}_{j,3}
  +2\phi^{(N,4)}_{j,4}+4\phi^{(N,4)}_{j,5}+2\phi^{(N,4)}_{j,6})\\ &\qquad
  +\frac{1}{r_k}(6\phi^{(N,4)}_{k,3}+2\phi^{(N,4)}_{k,4}+\phi^{(N,4)}_{k,5})
  +   \frac{1}{r_i}(6\phi^{(N,4)}_{i,1}+\phi^{(N,4)}_{i,5}+2\phi^{(N,4)}_{i,6})
  \\ &\qquad
  -  (\mathbf{n}_m)^T((\frac{2\mathbf{n}_j}{r_j}+\frac{6\mathbf{n}_m}{r_m})\phi^{(N,4)}_{m,2}+
  (\frac{\mathbf{n}_k}{r_k}+\frac{2\mathbf{n}_m}{r_m})\phi^{(N,4)}_{m,4}+  (\frac{\mathbf{n}_i}{r_i}+
  \frac{2\mathbf{n}_m}{r_m})\phi^{(N,4)}_{m,6}), \\
   \phi^{(F,4)}_{m,3} & = \tilde{\phi}^{(F,4)}_{m,3} -
   \frac{1}{r_k}(2\phi^{(N,4)}_{k,1}+6\phi^{(N,4)}_{k,2}+6\phi^{(N,4)}_{k,3}
   +4\phi^{(N,4)}_{k,4}+2\phi^{(N,4)}_{k,5}+2\phi^{(N,4)}_{k,6})\\ &\qquad
   +   \frac{1}{r_i}(6\phi^{(N,4)}_{i,2}+\phi^{(N,4)}_{i,4}+2\phi^{(N,4)}_{i,6})
   +\frac{1}{r_j}(6\phi^{(N,4)}_{j,1}+2\phi^{(N,4)}_{j,5}+\phi^{(N,4)}_{j,6})\\ &\qquad
   -  (\mathbf{n}_m)^T((\frac{2\mathbf{n}_k}{r_k}+\frac{6\mathbf{n}_m}{r_m})\phi^{(N,4)}_{m,3}+
   (\frac{\mathbf{n}_j}{r_j}+\frac{2\mathbf{n}_m}{r_m})\phi^{(N,4)}_{m,4}+  (\frac{\mathbf{n}_i}{r_i}+
   \frac{2\mathbf{n}_m}{r_m})\phi^{(N,4)}_{m,5}), }
where $m=1,\cdots,4.$ These 12 basis functions satisfy
$$
\mathbb{F}^{(4)}_{m,l}(\phi^{(F,4)}_{n,t})=\delta_{m,n}\delta_{l,t},~1\leq m,n\leq4,~1\leq l,t\leq 3,
$$
and vanish for the degrees of freedom of \eqref{degree11},\eqref{degree13} and \eqref{degree14}. Finally, we give last 18 basis functions which dual to the degrees of freedom of \eqref{degree12}.
\a{   \phi^{(E,4)}_{m,1}& = \tilde{\phi}^{(E,4)}_{m,1} +
     \frac{1}{r_i}(6\phi^{(N,4)}_{i,\tau_i(j)}+2\phi^{(N,4)}_{i,\tau_{i}(\tilde{i})}+2\phi^{(N,4)}_{i,\tau_i(\tilde{j})}
       +\phi^{(N,4)}_{i,\tau_{i}(j)+3}+2\phi^{(N,4)}_{i,\tau_{i}(\tilde{i})+3}\\ &\qquad
     +2\phi^{(N,4)}_{i,\tau_{i}(\tilde{j})+3})
     + \frac{2}{r_j}\phi^{N,4}_{j,\tau_j(i)}
      +(\frac{2\mathbf{n}_i}{r_i}+\frac{2\mathbf{n}_j}{r_j})^T(\phi^{(N,4)}_{\tilde{i},\tau_{\tilde{i}}(i)}\mathbf{n}_{\tilde{i}}
      +\phi^{N,4}_{\tilde{j},\tau_{\tilde{j}}(i)}\mathbf{n}_{\tilde{j}}),\\
    \phi^{(E,4)}_{m,2}& = \tilde{\phi}^{(E,4)}_{m,2} +
    \frac{1}{r_i}(12\phi^{(N,4)}_{i,\tau_i(j)}+2\phi^{(N,4)}_{i,\tau_{i}(\tilde{i})+3}+2\phi^{(N,4)}_{i,\tau_i(\tilde{j})+3})
    +(12\phi^{(N,4)}_{j,\tau_{j}(i)}+2\phi^{(N,4)}_{j,\tau_{j}(\tilde{i})+3}\\ &\qquad
    +2\phi^{(N,4)}_{j,\tau_{j}(\tilde{j})+3})   
    +(\frac{2\mathbf{n}_i}{r_i}+\frac{2\mathbf{n}_j}{r_j})^T(\phi^{(N,4)}_{\tilde{i},\tau_{\tilde{i}}(\tilde{j})+3}\mathbf{n}_{\tilde{j}}
    +\phi^{N,4}_{\tilde{j},\tau_{\tilde{j}}(\tilde{i})+3}\mathbf{n}_{\tilde{j}}),\\
     \phi^{(E,4)}_{m,3}& = \tilde{\phi}^{(E,4)}_{m,3} +
     \frac{1}{r_j}(6\phi^{(N,4)}_{j,\tau_j(i)}+2\phi^{(N,4)}_{j,\tau_{j}(\tilde{i})}+2\phi^{(N,4)}_{j,\tau_j(\tilde{j})}
     +\phi^{(N,4)}_{j,\tau_{j}(i)+3}+2\phi^{(N,4)}_{j,\tau_{j}(\tilde{i})+3}\\ &\qquad
     +2\phi^{(N,4)}_{i,\tau_{i}(\tilde{j})+3})
     + \frac{2}{r_i}\phi^{N,4}_{i,\tau_i(j)}
     +(\frac{2\mathbf{n}_i}{r_i}+\frac{2\mathbf{n}_j}{r_j})^T(\phi^{(N,4)}_{\tilde{i},\tau_{\tilde{i}}(j)}\mathbf{n}_{\tilde{i}}
     +\phi^{N,4}_{\tilde{j},\tau_{\tilde{j}}(j)}\mathbf{n}_{\tilde{j}}), }
 where $i$ and $j(i<j)$ form the $m$-th edge of the tetrahedron, $m=1,\cdots,6$. $\tilde{i}$ and $\tilde{j}(\tilde{i}<\tilde{j})$ are the other two index of vertices on the tetrahedron.(i.e. if $(i,j)=(2,4)$ then $(\tilde{i},\tilde{j}) = (1,3).$) They vanish for degrees of freedom of \eqref{degree12} to \eqref{degree14} and satisfy 
$$
\mathbb{E}^{(4)}_{m,l}(\phi^{(E,4)}_{n,t})=\delta_{m,n}\delta_{l,t},~1\leq m,n\leq 6,~1\leq l,t\leq 3.
$$

\subsection{The basis functions of $P_5+36P_7$ element}
In this section, let $\sigma_m(t,l),t\neq m,l\neq m, 1\leq t,l,m\leq 4$ be a indicator function, which represents the index of the basis function $\lambda_t^2\lambda_l$ on face $m$. For example, $\sigma_2(4,1)=6,\sigma_2(1,4)=9,\sigma_3(1,4)=7$. And the index of basis functions can be found in \eqref{index5}.  
The shape function space of the $P_5$+36$P_7$ element can be equivalently rewritten as
$$
	\tilde{\mathcal{P}}^+_5(T)=P_5(T)+\tilde{B}_5,
$$
	\begin{align*}
\tilde{B}_5=	&\mathrm{span}\{66\lambda_m^2\lambda_t^5-120\lambda_m^2\lambda_t^4,\\&99\lambda_m^2\lambda_t^5-165\lambda_m^3\lambda_t^4-135\lambda_m^2\lambda_t^4+180\lambda_m^3\lambda_t^3,\\&
	66\lambda_4\lambda_1^5\lambda_2-120\lambda_4\lambda_1^4\lambda_2,
	66\lambda_1^5\lambda_2\lambda_3-120\lambda_1^4\lambda_2\lambda_3,\\&
	66\lambda_1\lambda_2^5\lambda_3-120\lambda_1\lambda_2^4\lambda_3,
	66\lambda_2^5\lambda_3\lambda_4-120\lambda_2^4\lambda_3\lambda_4,\\&
	66\lambda_2\lambda_3^5\lambda_4-120\lambda_2\lambda_3^4\lambda_4,
	66\lambda_3^5\lambda_4\lambda_1-120\lambda_3^4\lambda_4\lambda_1,\\&
	66\lambda_3\lambda_4^5\lambda_1-120\lambda_3\lambda_4^4\lambda_1,
	66\lambda_4^5\lambda_1\lambda_2-120\lambda_4^4\lambda_1\lambda_2,\\&
	165\lambda_1^4\lambda_2^2\lambda_3-180\lambda_1^3\lambda_2^2\lambda_3,
	165\lambda_2^4\lambda_3^2\lambda_4-180\lambda_2^3\lambda_3^2\lambda_4,\\&
	165\lambda_3^4\lambda_4^2\lambda_1-180\lambda_3^3\lambda_4^2\lambda_1,
	165\lambda_4^4\lambda_1^2\lambda_2-180\lambda_4^3\lambda_1^2\lambda_2
	\},~1\leq m\neq t\leq 4.
	\end{align*}
Firstly, we show the following 40 basis functions dual to the degrees of freedom of \eqref{degree24}.
\a{
	\phi^{(N,5)}_{m,1}&=r_m( -10\tilde{\phi}^{(N,5)}_{m,0}-450\tilde{\phi}^{(N,5)}_{m,i}
	       +560\tilde{\phi}^{(N,5)}_{m,i+4}+180\tilde{\phi}^{(N,5)}_{m,j+8}+180\tilde{\phi}^{(N,5)}_{m,k+8}),\\
   \phi^{(N,5)}_{m,2}& =r_m( 80\tilde{\phi}^{(N,5)}_{m,0}-90\tilde{\phi}^{(N,5)}_{m,i}-540\tilde{\phi}^{(N,5)}_{m,j}-90\tilde{\phi}^{(N,5)}_{m,k}
               +560\tilde{\psi}^{(N,5)}_{m,j+4}-180\tilde{\phi}^{(N,5)}_{m,j+8}),\\
    \phi^{(N,5)}_{m,2}& =r_m( 80\tilde{\phi}^{(N,5)}_{m,0}-90\tilde{\phi}^{(N,5)}_{m,i}-90\tilde{\phi}^{(N,5)}_{m,j}
      -540\tilde{\phi}^{(N,5)}_{m,k}+560\tilde{\phi}^{(N,5)}_{m,k+4}-180\tilde{\phi}^{(N,5)}_{m,k+8}),\\
     \phi^{(N,5)}_{m,4}&=r_m( 60\tilde{\phi}^{(N,5)}_{m,0}-360\tilde{\phi}^{(N,5)}_{m,i}
       +90\tilde{\phi}^{(N,5)}_{m,j}-90\tilde{\phi}^{(N,5)}_{m,k}+180\tilde{\phi}^{(N,5)}_{m,j+8}\\ &\qquad\quad
      -2160\tilde{\phi}^{(N,5)}_{m,k+8}+5040\tilde{\phi}^{(N,5)}_{m,13})+r_j\tilde{b}_{j,1}+r_k\tilde{b}_{k,2},\\
    \phi^{(N,5)}_{m,5}&=r_m( 60\tilde{\phi}^{(N,5)}_{m,0}+4680\tilde{\phi}^{(N,5)}_{m,i}-90\tilde{\phi}^{(N,5)}_{m,j}
         +90\tilde{\phi}^{(N,5)}_{m,k}-5040\tilde{\phi}^{(N,5)}_{m,i+4}\\  &\qquad\quad
      - 2160\tilde{\phi}^{(N,5)}_{m,j+8}+
       180\tilde{\phi}^{(N,5)}_{m,k+8}- 5040\tilde{\phi}^{(N,5)}_{m,13})-r_j\tilde{b}_{j,1}-r_k\tilde{b}_{k,2},\\
 \phi^{(N,5)}_{m,6}&=r_m( -180\tilde{\phi}^{(N,5)}_{m,0}-1530\tilde{\phi}^{(N,5)}_{m,i}+3240\tilde{\phi}^{(N,5)}_{m,j}
      -1350\tilde{\phi}^{(N,5)}_{m,k}+1680\tilde{\phi}^{(N,5)}_{m,i+4}-3360\tilde{\phi}^{(N,5)}_{m,j+4}\\ &\qquad\quad
      + 1680\tilde{\phi}^{(N,5)}_{m,k+4}+2160\tilde{\phi}^{(N,5)}_{m,j+8}
        -2700\tilde{\phi}^{(N,5)}_{m,k+8}+5040\tilde{\phi}^{(N,5)}_{m,13})+r_j\tilde{b}_{j,1}+r_k\tilde{b}_{k,2},\\
         \phi^{(N,5)}_{m,7}&=r_m( -690\tilde{\phi}^{(N,5)}_{m,0}+2520\tilde{\phi}^{(N,5)}_{m,i}
         +2070\tilde{\phi}^{(N,5)}_{m,j}+2340\tilde{\phi}^{(N,5)}_{m,k}-1680\tilde{\phi}^{(N,5)}_{m,i+4}
         -1680\tilde{\phi}^{(N,5)}_{m,j+4}\\ &\qquad\quad
         -1680\tilde{\phi}^{(N,5)}_{m,k+4}-180\tilde{\phi}^{(N,5)}_{m,j+8}+2700\tilde{\phi}^{(N,5)}_{m,k+8}-5040\tilde{\phi}^{(5)}_{m,13})
         -r_j\tilde{b}_{j,1}-r_k\tilde{b}_{k,2},\\
    \phi^{(N,5)}_{m,8}&=r_m( -690\tilde{\phi}^{(N,5)}_{m,0}-2520\tilde{\phi}^{(N,5)}_{m,i}+2340\tilde{\phi}^{(N,5)}_{m,j}
        +2070\tilde{\phi}^{(N,5)}_{m,k}+3360\tilde{\phi}^{(N,5)}_{m,i+4}-1680\tilde{\phi}^{(N,5)}_{m,j+4}\\ &\qquad\quad
     -1680\tilde{\phi}^{(N,5)}_{m,k+4}+2700\tilde{\phi}^{(N,5)}_{m,j+8}
    -180\tilde{\phi}^{(N,5)}_{m,k+8}+5040\tilde{\phi}^{(N,5)}_{m,13})
       +r_j\tilde{b}_{j,1}+r_k\tilde{b}_{k,2},\\
    \phi^{(N,5)}_{m,9}&=r_m( -180\tilde{\phi}^{(N,5)}_{m,0}+3510\tilde{\phi}^{(N,5)}_{m,i}-1350\tilde{\phi}^{(N,5)}_{m,j}
  +3240\tilde{\phi}^{(N,5)}_{m,k}-3360\tilde{\phi}^{(N,5)}_{m,i+4}+1680\tilde{\phi}^{(N,5)}_{m,j+4}\\ &\qquad\quad
   -3360\tilde{\phi}^{(N,5)}_{m,k+4}-2700\tilde{\phi}^{(N,5)}_{m,j+8}+2160\tilde{\phi}^{(N,5)}_{m,k+8}-5040\tilde{\phi}^{(5)}_{m,13})
      -r_j\tilde{b}_{j,1}-r_k\tilde{b}_{k,2},\\
     \phi^{(N,5)}_{m,10}&=r_m( 2400\tilde{\phi}^{(N,5)}_{m,0}-8820\tilde{\phi}^{(N,5)}_{m,i}-8820\tilde{\phi}^{(N,5)}_{m,j}
     -8820\tilde{\phi}^{(N,5)}_{m,k}+6720\tilde{\phi}^{(N,5)}_{m,i+4}+6720\tilde{\phi}^{(N,5)}_{m,j+4}\\ &\qquad\quad
       + 6720\tilde{\phi}^{(N,5)}_{m,k+4}),
  } where $m=1,\cdots,4$ and
\begin{align*}
  \tilde{b}_{m,1}&=(240\tilde{\phi}^{(N,5)}_{m,0}-270\tilde{\phi}^{(N,5)}_{m,i}-1620\tilde{\phi}^{(N,5)}_{m,j}-270\tilde{\phi}^{(N,5)}_{m,k}
              +1680\tilde{\phi}^{(N,5)}_{m,j+4}-540\tilde{\phi}^{(N,5)}_{m,j+8}),\\
   \tilde{b}_{m,2}&=(30\tilde{\phi}^{(N,5)}_{m,0}+1350\tilde{\phi}^{(N,5)}_{m,i}-1680\tilde{\phi}^{(N,5)}_{m,i+4}-540\tilde{\phi}^{(N,5)}_{m,j+8}
       -540\tilde{\phi}^{(N,5)}_{m,k+8}).
\end{align*}
They vanish for degrees of freedom of \eqref{degree21} to \eqref{degree23} and satisfy
$$
\mathbb{N}^{(5)}_{m,l}(\phi^{(N,5)}_{n,t})=\delta_{m,n}\delta_{l,t},~1\leq m,n\leq 4,~1\leq l,t\leq 10.
$$
 Then, four basis functions dual to degrees of freedom of \eqref{degree23}, which vanish for degrees of freedom of \eqref{degree21},\eqref{degree22} and \eqref{degree24}, are shown as
\begin{align*}
\phi^{(T,5)}_{m}= \tilde{\phi}^{(T,5)}_{m}& -\frac{4}{3r_m}(6\phi^{(N,5)}_{m,1}+6\phi^{(N,5)}_{m,2}+6\phi^{(N,5)}_{m,3} \\
+& 3\phi^{(N,5)}_{m,4}+3\phi^{(N,5)}_{m,5}+3\phi^{(N,5)}_{m,6}+3\phi^{(N,5)}_{m,7}
+3\phi^{(N,5)}_{m,8}+3\phi^{(N,5)}_{m,9}+2\phi^{(N,5)}_{m,10})
\\+&\frac{2}{3r_i}(18\phi^{(N,5)}_{i,\tau_i(m)}+6\phi^{(N,5)}_{i,\sigma_i(\tau_i(m),\tau_i(j))}+6\phi^{(N,5)}_{i,\sigma_i(\tau_i(m),\tau_i(k))}+
3\phi^{(N,5)}_{i,\sigma_i(\tau_i(j),\tau_i(m))}\\+&3\phi^{(N,5)}_{i,\sigma_i(\tau_i(k),\tau_i(m))}+2\phi^{(N,5)}_{i,10})+
\frac{2}{3r_j}(18\phi^{(N,5)}_{j,1}+6\phi^{(N,5)}_{j,\sigma_j(\tau_j(m),\tau_j(i))}\\+&6\phi^{(N,5)}_{j,\sigma_j(\tau_j(m),\tau_j(k))}+
3\phi^{(N,5)}_{j,\sigma_j(\tau_j(i),\tau_j(m))}+3\phi^{(N,5)}_{j,\sigma_j(\tau_j(k),\tau_j(m))}+2\phi^{(N,5)}_{j,10})
\\+&\frac{2}{3r_k}(18\phi^{(N,5)}_{k,1}+6\phi^{(N,5)}_{k,\sigma_k(\tau_k(m),\tau_k(i))}+6\phi^{(N,5)}_{k,\sigma_k(\tau_k(m),\tau_k(j))}+
3\phi^{(N,5)}_{k,\sigma_k(\tau_k(i),\tau_k(m))}\\+&3\phi^{(N,5)}_{k,\sigma_k(\tau_k(j),\tau_k(m))}+2\phi^{(N,5)}_{k,10}),~m=1,\cdots,4,
\end{align*}
and they satisfy
$$
\mathbb{T}^{(5)}_{m}(\phi^{(T,5)}_{t})=\delta_{m,t},~1\leq m,t\leq 4.
$$
Next, we present the following 24 basis functions associated to degrees of freedom of \eqref{degree22}.
\begin{align*}
\phi^{(F,5)}_{m,1}= \tilde{\phi}^{(F,5)}_{m,1}& +\frac{1}{r_i}(12\phi^{(N,5)}_{i,1}
      +12\phi^{(N,5)}_{i,2}+3\phi^{(N,5)}_{i,3}
     +6\phi^{(N,5)}_{i,4}+3\phi^{(N,5)}_{i,5}+3\phi^{(N,5)}_{i,6}+6\phi^{(N,5)}_{i,7}
         \\&  +2\phi^{(N,5)}_{i,8}+2\phi^{(N,5)}_{i,9}+2\phi^{(N,5)}_{i,10})
       + \frac{1}{r_j}(12\phi^{(N,5)}_{j,3}+2\phi^{(N,5)}_{j,8}+\phi^{(N,5)}_{j,9})\\ &
       +\frac{1}{r_k}(12\phi^{(N,5)}_{k,2}+2\phi^{(N,5)}_{k,6}+\phi^{(N,5)}_{k,7})
+\frac{(\mathbf{n}_i^T\mathbf{n}_m)}{r_i}(9\phi^{(N,5)}_{m,1}+2\phi^{(N,5)}_{m,4}+2\phi^{(N,5)}_{m,5})
   \\ &
  + \frac{(\mathbf{n}_j^T\mathbf{n}_m)}{r_j}(12\phi^{(N,5)}_{m,1}+\phi^{(N,5)}_{m,4}+2\phi^{(N,5)}_{m,5})
  +\frac{(\mathbf{n}_k^T\mathbf{n}_m)}{r_k}(12\phi^{(N,5)}_{m,1}+2\phi^{(N,5)}_{m,4}+\phi^{(N,5)}_{m,5}),
\end{align*}
\begin{align*}
    \phi^{(F,5)}_{m,2}= \tilde{\phi}^{(F,5)}_{m,2}&
      +\frac{1}{r_j}(12\phi^{(N,5)}_{j,1}+3\phi^{(N,5)}_{j,2}+12\phi^{(N,5)}_{j,3}
      +3\phi^{(N,5)}_{j,4}+6\phi^{(N,5)}_{j,5}+2\phi^{(N,5)}_{j,6}\\&
    +2\phi^{(N,5)}_{j,7}+6\varphi^{N,5}_{j,8}+3\phi^{(N,5)}_{j,9}+2\phi^{(N,5)}_{j,10})
      + \frac{1}{r_i}(12\phi^{(N,5)}_{i,1}+2\phi^{(N,5)}_{i,4}+\phi^{(N,5)}_{i,5}) \\&
  +\frac{1}{r_k}(12\phi^{(N,5)}_{k,3}+\phi^{(N,5)}_{k,8}+2\phi^{(N,5)}_{k,9})
   +\frac{(\mathbf{n}_j^T\mathbf{n}_m)}{r_j}(9\phi^{(N,5)}_{m,2}+2\phi^{(N,5)}_{m,6}+2\phi^{(N,5)}_{m,7})\\
   &+ \frac{(\mathbf{n}_i^T\mathbf{n}_m)}{r_i}(12\phi^{(N,5)}_{m,2}+2\phi^{(N,5)}_{m,6}+\phi^{(N,5)}_{m,7})
    +\frac{(\mathbf{n}_k^T\mathbf{n}_m)}{r_k}(12\phi^{(N,5)}_{m,2}+\phi^{(N,5)}_{m,6}
    +2\phi^{(N,5)}_{m,7}),
\end{align*}
\begin{align*}
\phi^{(F,5)}_{m,3}= \tilde{\phi}^{(F,5)}_{m,3}&
        +\frac{1}{r_k}(3\phi^{(N,5)}_{k,1}+12\phi^{(N,5)}_{k,2}+12\phi^{(N,5)}_{k,3}
        +2\phi^{(N,5)}_{k,4}+2\phi^{(N,5)}_{k,5}+6\phi^{(N,5)}_{k,6} \\&
     +3\phi^{(N,5)}_{k,7}+3\phi^{(N,5)}_{k,8}+6\phi^{(N,5)}_{k,9}+2\phi^{(N,5)}_{k,10})
    + \frac{1}{r_i}(12\phi^{(N,5)}_{i,2}+\phi^{(N,5)}_{i,6}+2\phi^{(N,5)}_{i,7}) \\&
   +\frac{1}{r_j}(12\phi^{(N,5)}_{j,2}+\phi^{(N,5)}_{j,4}+2\phi^{(N,5)}_{j,5})
   +\frac{(\mathbf{n}_k^T\mathbf{n}_m)}{r_k}(9\phi^{(N,5)}_{m,3}+2\phi^{(N,5)}_{m,8}+2\phi^{(N,5)}_{m,9})\\
  &+ \frac{(\mathbf{n}_i^T\mathbf{n}_m)}{r_i}(12\phi^{(N,5)}_{m,3}+\phi^{(N,5)}_{m,8}
     +2\phi^{(N,5)}_{m,9})
    +\frac{(\mathbf{n}_j^T\mathbf{n}_m)}{r_j}(12\phi^{(N,5)}_{m,3}+2\phi^{(N,5)}_{m,8}
     +\phi^{(N,5)}_{m,9}),
\end{align*}
\begin{align*}
    \phi^{(F,5)}_{m,4}= \tilde{\phi}^{(F,5)}_{m,4}&
          -\frac{1}{r_i}(36\phi^{(N,5)}_{i,1}+12\phi^{(N,5)}_{i,4}+6\phi^{(N,5)}_{i,5}
         +6\phi^{(N,5)}_{i,7}+2\phi^{(N,5)}_{i,8}+2\phi^{(N,5)}_{i,10})
    \\& +\frac{1}{r_k}(6\phi^{(N,5)}_{k,7}+6\phi^{(N,5)}_{k,9}+\phi^{(N,5)}_{k,10})
   - \frac{1}{r_j}(36\phi^{(N,5)}_{j,3}+6\phi^{(N,5)}_{j,5}+2\phi^{(N,5)}_{j,6}
   \\& +12\phi^{(N,5)}_{j,8}+6\phi^{(N,5)}_{j,9}+2\phi^{(N,5)}_{j,10})
    +\frac{(\mathbf{n}_i^T\mathbf{n}_m)}{r_i}(4\phi^{(N,5)}_{m,4}+6\phi^{(N,5)}_{m,7}
         +2\phi^{(N,5)}_{m,10}) \\&
   + \frac{(\mathbf{n}_j^T\mathbf{n}_m)}{r_j}(6\phi^{(N,5)}_{m,4}+4\phi^{(N,5)}_{m,7}
    +2\phi^{(N,5)}_{m,10})
     +\frac{(\mathbf{n}_k^T\mathbf{n}_m)}{r_k}(6\phi^{(N,5)}_{m,4}+6\phi^{(N,5)}_{m,7}
     +\phi^{(N,5)}_{m,10}),
\end{align*}
\begin{align*}
\phi^{(F,5)}_{m,5}= \tilde{\phi}^{(F,5)}_{m,5}&
             -\frac{1}{r_i}(36\phi^{(N,5)}_{i,2}+6\phi^{(N,5)}_{i,4}+6\phi^{(N,5)}_{i,6}
      +12\phi^{(N,5)}_{i,7}+2\phi^{(N,5)}_{i,9}+2\phi^{(N,5)}_{i,10})\\&
     +\frac{1}{r_j}(6\phi^{(N,5)}_{j,5}+6\phi^{(N,5)}_{j,8}+\phi^{(N,5)}_{j,10})
      - \frac{1}{r_k}(36\phi^{(N,5)}_{k,2}+2\phi^{(N,5)}_{k,4}+12\phi^{(N,5)}_{k,6}
          +6\phi^{(N,5)}_{k,7}\\&
            +6\phi^{(N,5)}_{k,9}+2\phi^{(N,5)}_{k,10})
      +\frac{(\mathbf{n}_i^T\mathbf{n}_m)}{r_i}(4\phi^{(N,5)}_{m,5}
        +6\phi^{(N,5)}_{m,8}+2\phi^{(N,5)}_{m,10})\\
    &+ \frac{(\mathbf{n}_j^T\mathbf{n}_m)}{r_j}(6\phi^{(N,5)}_{m,5}+6\phi^{(N,5)}_{m,8}
      +\phi^{(N,5)}_{m,10})
      +\frac{(\mathbf{n}_k^T\mathbf{n}_m)}{r_k}(6\phi^{(N,5)}_{m,5}+4\phi^{(N,5)}_{m,8}
      +2\phi^{(N,5)}_{m,10}),
\end{align*}
\begin{align*}
\phi^{(F,5)}_{m,6}= \tilde{\phi}^{(F,5)}_{m,6}
    & -\frac{1}{r_j}(36\phi^{(N,5)}_{j,1}+6\phi^{(N,5)}_{j,4}+12\phi^{(N,5)}_{j,5}
     +2\phi^{(N,5)}_{j,7}+2\phi^{(N,5)}_{j,8}+2\phi^{(N,5)}_{j,10})
       \\ & +\frac{1}{r_i}(6\phi^{(N,5)}_{i,4}+6\phi^{(N,5)}_{i,6}+\phi^{(N,5)}_{i,10})
        - \frac{1}{r_k}(36\phi^{(N,5)}_{k,3}+2\phi^{(N,5)}_{k,5}+6\phi^{(N,5)}_{k,6}
       \\& +6\phi^{(N,5)}_{k,8}+12\phi^{(N,5)}_{k,9}+2\phi^{(N,5)}_{k,10})
    +\frac{(\mathbf{n}_i^T\mathbf{n}_m)}{r_i}(6\phi^{(N,5)}_{m,6}+6\phi^{(N,5)}_{m,9}
     +\phi^{(N,5)}_{m,10})\\
    & + \frac{(\mathbf{n}_j^T\mathbf{n}_m)}{r_j}(4\phi^{(N,5)}_{m,6}+6\phi^{(N,5)}_{m,9}
      +2\phi^{(N,5)}_{m,10})
      +\frac{(\mathbf{n}_k^T\mathbf{n}_m)}{r_k}(6\phi^{(N,5)}_{m,6}+4\phi^{(N,5)}_{m,9}
        +2\phi^{(N,5)}_{m,10}),
\end{align*}
where $m=1,\cdots,4.$
They vanish for degrees of freedom of \eqref{degree21},\eqref{degree23} and \eqref{degree24} and satisfy
$$
\mathbb{F}^{(5)}_{m,l}(\phi^{(F,5)}_{n,t})=\delta_{m,n}\delta_{l,t},~1\leq m,n\leq 4,~1\leq l,t\leq 6.
$$
The rest 24 basis functions, which vanish for degrees of freedom of \eqref{degree22} to \eqref{degree24} are shown as
\begin{align*}
\phi^{(E,5)}_{m,1}= \tilde{\phi}^{(E,5)}_{m,1}& -\frac{1}{3r_i}(24\phi^{(N,5)}_{i,\tau_i(j)}
        +6\varphi^{N,5}_{i,\tau_i(\tilde{i})}+6\phi^{(N,5)}_{i,\tau_i(\tilde{j})}+6\phi^{(N,5)}_{i,2(\tau_i(j)+1)}
          +6\phi^{(N,5)}_{i,2(\tau_i(j)+1)+1} \\ &
          +2\phi^{(N,5)}_{i,2(\tau_i(\tilde{i})+1)+\mathrm{sgn}(\tau_i(\tilde{i})-\tau_i(\tilde{j}))}+4\phi^{(N,5)}_{i,2(\tau_i(\tilde{i})+1)+\mathrm{sgn}(\tau_i(\tilde{j})-\tau_i(\tilde{i}))}\\&+2\phi^{(N,5)}_{i,2(\tau_i(\tilde{j})+1)+\mathrm{sgn}(\tau_i(\tilde{j})-\tau_i(\tilde{i}))}
             +4\phi^{(N,5)}_{i,2(\tau_i(\tilde{j})+1)+\mathrm{sgn}(\tau_i(\tilde{i})-\tau_i(\tilde{j}))}+2\phi^{(N,5)}_{i,10})\\ &
       +\frac{2}{r_j}\phi_{j,\tau_j(i)}
         +2(\frac{\mathbf{n}_i^T}{r_i}
    +\frac{\mathbf{n}_j^T}{r_j})(\mathbf{n}_{\tilde{i}}\phi^{(N,5)}_{{\tilde{i}},\tau_{\tilde{i}}(i)}+\mathbf{n}_{\tilde{j}}\phi^{(N,5)}_{{\tilde{j}},\tau_{\tilde{j}}(i)}), \\
   \phi^{(E,5)}_{m,2}=\tilde{\phi}^{(E,5)}_{m,2}& -\frac{1}{3r_j}(24\phi^{(N,5)}_{j,\tau_j(i)}
   +6\varphi^{N,5}_{j,\tau_j(\tilde{i})}+6\phi^{(N,5)}_{j,\tau_j(\tilde{j})}+6\phi^{(N,5)}_{j,2(\tau_j(i)+1)}
   +6\phi^{(N,5)}_{j,2(\tau_j(i)+1)+1} \\ &
   +2\phi^{(N,5)}_{j,2(\tau_j(\tilde{i})+1)+\mathrm{sgn}(\tau_j(\tilde{i})-\tau_j(\tilde{j}))}+4\phi^{(N,5)}_{j,2(\tau_j(\tilde{i})+1)+\mathrm{sgn}(\tau_j(\tilde{j})-\tau_j(\tilde{i}))}\\&+2\phi^{(N,5)}_{j,2(\tau_j(\tilde{j})+1)+\mathrm{sgn}(\tau_j(\tilde{j})-\tau_j(\tilde{i}))}
   +4\phi^{(N,5)}_{j,2(\tau_j(\tilde{j})+1)+\mathrm{sgn}(\tau_j(\tilde{i})-\tau_j(\tilde{j}))}+2\phi^{(N,5)}_{j,10})\\ &
   +\frac{2}{r_i}\phi_{i,\tau_i(j)}
   +2(\frac{\mathbf{n}_i^T}{r_i}
   +\frac{\mathbf{n}_j^T}{r_j})(\mathbf{n}_{\tilde{i}}\phi^{(N,5)}_{{\tilde{i}},\tau_{\tilde{i}}(j)}+\mathbf{n}_{\tilde{j}}\phi^{(N,5)}_{{\tilde{j}},\tau_{\tilde{j}}(j)}),
\end{align*}
\begin{align*} \phi^{(E,5)}_{m,3}= \tilde{\phi}^{(E,5)}_{m,3}
    & +\frac{1}{r_i}(36\phi^{(N,5)}_{i,\tau_i(j)}+6\phi^{(N,5)}_{i,2(\tau_i(j)+1)}
         +6\phi^{(N,5)}_{i,2(\tau_i(j)+1)+1}+2\phi^{(N,5)}_{i,\sigma_i(\tilde{i},j)}+2\phi^{(N,5)}_{i,\sigma_i(\tilde{j},j)}
    \\& +\phi^{(N,5)}_{i,10}) -\frac{1}{r_j}(24\phi^{(N,5)}_{j,\tau_j(i)}
           +2\phi^{(N,5)}_{j,2(\tau_j(i)+1)}+2\phi^{(N,5)}_{j,2(\tau_j(i)+1)+1})
      \\ & +2(\frac{\mathbf{n}_i^T}{r_i}+\frac{\mathbf{n}_j^T}{r_j})
     (\mathbf{n}_{\tilde{i}}\phi^{(N,5)}_{{\tilde{i}},\sigma_{\tilde{i}}(i,j)}+\mathbf{n}_{\tilde{j}}\phi^{(N,5)}_{{\tilde{j}},\sigma_{\tilde{j}}(i,j)}), \\
     \phi^{(E,5)}_{m,4}= \tilde{\phi}^{(E,5)}_{m,4}
       & +\frac{1}{r_j}(36\phi^{(N,5)}_{j,\tau_j(i)}+6\phi^{(N,5)}_{j,2(\tau_j(i)+1)}
       +6\phi^{(N,5)}_{j,2(\tau_j(i)+1)+1}+2\phi^{(N,5)}_{j,\sigma_j(\tilde{i},i)}+2\phi^{(N,5)}_{j,\sigma_j(\tilde{j},i)}
       \\& +\phi^{(N,5)}_{j,10}) -\frac{1}{r_i}(24\phi^{(N,5)}_{i,\tau_i(j)}
       +2\phi^{(N,5)}_{i,2(\tau_i(j)+1)}+2\phi^{(N,5)}_{i,2(\tau_i(j)+1)+1})
       \\ & +2(\frac{\mathbf{n}_i^T}{r_i}+\frac{\mathbf{n}_j^T}{r_j})
       (\mathbf{n}_{\tilde{i}}\phi^{(N,5)}_{{\tilde{i}},\sigma_{\tilde{i}}(j,i)}+\mathbf{n}_{\tilde{j}}\phi^{(N,5)}_{{\tilde{j}},\sigma_{\tilde{j}}(j,i)}),
\end{align*}
where $i$ and $j(i<j)$ form the $m$-th edge of the tetrahedron, $m=1,\cdots,6$. $\tilde{i}$ and $\tilde{j}(\tilde{i}<\tilde{j})$ are the other two index of vertices on the tetrahedron.
And those 24 functions satisfy
$$
\mathbb{E}^{(5)}_{m,l}(\phi^{(E,5)}_{n,t})=\delta_{m,n}\delta_{l,t},~1\leq m,n\leq 6,~1\leq l,t\leq 4.
$$

% \printbibliography


\begin{thebibliography}{999}
\bibitem{Soblev}{Adams R A, Fournier J J F,}
{ Sobolev spaces[M],}
Elsevier, 2003.
\bibitem{Alfeld-Sirvent} P. Alfeld and M. Sirvent,
  The structure of multivariate superspline spaces of
	high degree,
    Math. Comp. 57 (1991), no. 195, 299--308.




 \bibitem{Argyris} J. H. Argyris, I. Fried and D. W. Scharpf,
    { The TUBA family of plate elements for the matrix
     displacement method,} Aeronautical Journal, 72 (1968),  514--517.

\bibitem{Brenner-Scott} S. C. Brenner and L. R. Scott,
 The mathematical theory of finite element methods.
    Third edition. Texts in Applied Mathematics,
     15. Springer, New York, 2008.


%\bibitem{Chen-Chen}  H. R. Chen and S. C. Chen,
%   C0-nonconforming elements for a fourth-order elliptic problem,
%   (Chinese) Math. Numer. Sin. 35 (2013), no. 1, 21--30.
\bibitem{Chen-Chen-Qiao} {Chen, Hongru and Chen, Shaochun and Qiao, Zhonghua,}
{{$C^0$}-nonconforming tetrahedral and cuboid elements for the
	three-dimensional fourth order elliptic problem,}
Numer. Math., 124(2013), no.1, 99--119. 

\bibitem{Ciarlet} { P.G. Ciarlet, }
  { The finite element method for elliptic problems,  }
  North-Holland,   Amsterdam,    1978.


\bibitem{Fraeijs} B. Fraeijs and De Veubeke,
  Variational principles and the patch test,
   Intern. J. Numer. Meth. Eng.,  8 (1975),  783--801.

\bibitem{Gao-Zhang-Wang}
   B. Gao, S. Zhang and M. Wang,
  A note on the nonconforming finite elements for elliptic problems,
  J. Comput. Math. 29 (2011), no. 2, 215--226.


\bibitem{Guzman} J. Guzman, D.  Leykekhman and M. Neilan,
  A family of non-conforming elements and the analysis of Nitsche's method for a singularly perturbed fourth order problem,
   Calcolo 49 (2012), no. 2, 95--125.
\bibitem{Hu-Ma-Shi} J. Hu,  R. Ma, and Z. Shi,
  A new a priori error estimate of nonconforming finite element methods,
   Sci. China Math. 57 (2014), no. 5, 887--902.


\bibitem{Hu-Huang-Zhang} J. Hu, Y. Huang and S. Zhang,
   The lowest order
    differentiable finite element on rectangular grids,
     SIAM Numer. Anal.  49 (2011), No 4, 1350--1368.

\bibitem{Hu-Zhang} J. Hu and S. Zhang,
  The minimal conforming $H^k$ finite element spaces on
   $R^n$ rectangular grids,
   Math. Comp., 84 (2015) no. 292, 563--579.

 \bibitem{Hu-Zhang-T} J. Hu and S. Zhang,
  A canonical construction of $H^m$-nonconforming triangular finite elements,
    Ann. Appl. Math. 33 (2017) no. 33,  266--288.

 \bibitem{Hu-Zhang-3D} J. Hu and S. Zhang,
 Constructions of nonconforming finite elements for fourth order
    elliptic problems
  and applications,  Journal of Computational Mathematics, accepted.

\bibitem{Powell} { M.J.D. Powell and M.A. Sabin},
  {\it }{ Piecewise quadratic approximations on triangles},
    ACM Transactions on Mathematical   Software,
    3-4 (1977), 316--325.

\bibitem{Scott-Zhang}
L.R. Scott and S. Zhang,
   Finite element interpolation of nonsmooth functions satisfying boundary conditions,
    Math. Comp. 54 (1990), 483--493.

\bibitem{ShiWang}
  Z. Shi and M. Wang,
  Mathematical theory of some nonstandard finite element methods.
  Computational mathematics in China, 111--125, Contemp. Math., 163,
  Amer. Math. Soc., Providence, RI, 1994.

\bibitem{Shi-Wang}
  Z. Shi and M. Wang,
  Finite element methods, Science Press, 2013.


\bibitem{Wang-Shi-Xu} M. Wang, Z. Shi and J. Xu,
   A new class of Zienkiewicz-type non-conforming element in any dimensions,
   Numer. Math., 106 (2007) no. 2,  335--347.

\bibitem{Wang-Xu} M. Wang and J. Xu,
  Minimal finite element spaces for 2$m$-th-order partial differential
  equations in Rn. Math. Comp. 82 (2013), no. 281, 25--43.



\bibitem{Wang-Zu-Zhang}
   M. Wang, P. Zu and S. Zhang,
 High accuracy nonconforming finite elements for fourth order problems,
     Sci. China Math. 55 (2012), no. 10, 2183--2192.

\bibitem{Zhang-C1} S. Zhang,
  A family of 3D continuously differentiable finite elements
           on tetrahedral grids, Appl. Numer. Math.,
          59 (2009), no. 1,   219--233.


\bibitem{ZhangC1Qk} S. Zhang,  On the full $C_1$-$Q_k$ %25
          finite element spaces on rectangles and cuboids,
       Adv. Appl. Math. Mech., 2 (2010), 701--721.

\bibitem{Zhang-4D} S. Zhang, A family of differentiable finite elements
          on simplicial grids in four space dimensions,
        Math. Numer. Sin.  38 (2016), no. 3, 309--324.


\end{thebibliography}
\end{document}